\documentclass{amsart}
\usepackage{graphicx}
\usepackage{amssymb}
\usepackage{amsfonts}
\swapnumbers
\newtheorem{theorem}{Theorem}[section]
\newtheorem{lemma}[theorem]{Lemma}
\newtheorem{corollary}[theorem]{Corollary}

\theoremstyle{definition}

\newtheorem{remark}[theorem]{Remark}

\numberwithin{equation}{section}
 \theoremstyle{plain}
 
 \numberwithin{equation}{section} 
 \numberwithin{figure}{section} 
 \theoremstyle{plain}
 \theoremstyle{remark}
 \newtheorem*{acknowledgement*}{Acknowledgement}

\newcommand{\cB}{{\mathcal B}}

\newcommand{\cF}{{\mathcal F}}
\newcommand{\cG}{{\mathcal G}}
\newcommand{\cH}{{\mathcal H}}
\newcommand{\cI}{{\mathcal I}}
\newcommand{\cJ}{{\mathcal J}}

\newcommand{\cL}{{\mathcal L}}

\newcommand{\cQ}{{\mathcal Q}}
\newcommand{\cR}{{\mathcal R}}

\newcommand{\cT}{{\mathcal T}}

\newcommand{\Te}{{\Theta}}

\newcommand{\vt}{{\vartheta}}
\newcommand{\Om}{{\Omega}}
\newcommand{\om}{{\omega}}
\newcommand{\ve}{{\varepsilon}}
\newcommand{\del}{{\delta}}
\newcommand{\Del}{{\Delta}}
\newcommand{\gam}{{\gamma}}
\newcommand{\Gam}{{\Gamma}}

\newcommand{\vr}{{\varrho}}
\newcommand{\Sig}{{\Sigma}}
\newcommand{\sig}{{\sigma}}
\newcommand{\al}{{\alpha}}
\newcommand{\be}{{\beta}}
\newcommand{\ka}{{\kappa}}

\newcommand{\up}{{\upsilon}}
\newcommand{\Up}{{\Upsilon}}

\newcommand{\im}{{\imath}}

\newcommand{\bbR}{{\mathbb R}}

\newcommand{\bbI}{{\mathbb I}}

\begin{document}
\title[]{Strong diffusion approximation in averaging\\
and value computation in Dynkin's games}%
 \vskip 0.1cm
 \author{ Yuri Kifer\\
\vskip 0.1cm
 Institute  of Mathematics\\
Hebrew University\\
Jerusalem, Israel}%
\address{
Institute of Mathematics, The Hebrew University, Jerusalem 91904, Israel}
\email{ kifer@math.huji.ac.il}%

\thanks{ }
\subjclass[2000]{Primary: 34C29 Secondary: 60F15, 60G40, 91A05}%
\keywords{averaging, diffusion approximation, $\phi$-mixing,
Dynkin game, stationary process.}%
\dedicatory{  }
 \date{\today}
\begin{abstract}\noindent
It is known since \cite{Kha} that the slow motion $X^\ve$ in the time-scaled multidimensional
averaging setup $\frac {dX^\ve(t)}{dt}=\frac 1\ve
B(X^\ve(t),\,\xi(t/\ve^2))+b(X^\ve(t),\,\xi(t/\ve^2)),\, t\in [0,T]$
converges weakly as $\ve\to 0$ to a diffusion process provided $EB(x,\xi(s))\equiv 0$ where $\xi$
is a sufficiently fast mixing stochastic process. In this paper we show that both $X^\ve$ and a family of diffusions $\Xi^\ve$ can be redefined on a common sufficiently rich probability space
so that $E\sup_{0\leq t\leq T}|X^\ve(t)-\Xi^\ve(t)|^{2M}\leq C(M)\ve^\del$ for some $C(M),\del>0$
 and all $M\ge 1,\,\ve>0$,
 where all $\Xi^\ve,\, \ve>0$ have the same diffusion coefficients but underlying
Brownian motions may change with $\ve$. This is the first strong approximation result both in the above setup and at all when the limit is a nontrivial multidimensional diffusion. We obtain also a similar result for the corresponding discrete time averaging setup which was not considered before
at all. As an application we consider Dynkin's games with path dependent payoffs involving
a diffusion and obtain error estimates for computation of values of such games by means of such
 discrete time approximations which provides a more effective computational tool than the
 standard discretization of the diffusion  itself.
\end{abstract}
\maketitle
\markboth{Yu.Kifer}{Strong diffusion approximation}
\renewcommand{\theequation}{\arabic{section}.\arabic{equation}}
\pagenumbering{arabic}

\section{Introduction}\label{sec1}\setcounter{equation}{0}
This paper is motivated by two separate lines of research: weak diffusion limits in time scaled
 averaging setups \cite{Kha}, \cite{PK}, \cite{Bor}, \cite{CE}, \cite{BF} and multidimensional strong approximation
theorems \cite{BP}, \cite{KP}, \cite{MP1}, \cite{DP} etc. Namely, we will deal with systems of ordinary
differential equations of the form
\begin{equation}\label{1.1}
\frac {dX^\ve(t)}{dt}=\frac 1\ve B(X^\ve(t),\xi(t/\ve^2))+b(X^\ve(t),\,\xi(t/\ve^2)),\,\, t\in[0,T]
\end{equation}
where $B(\cdot,\xi(s))$ and $b(\cdot,\xi(s))$ are (random) Lipschitz continuous vector fields on
$\bbR^d$ and $\xi$ is a sufficiently fast mixing stationary process which is viewed as a fast motion
while $X^\ve$ moves slower. In the classical averaging setup the fast motion is usually considered on
the time scale $1/\ve$ but here we assume that
\begin{equation}\label{1.2}
EB(x,\xi(s))\equiv 0
\end{equation}
and in order to detect an interesting behavior of the slow motion $X^\ve$ the time scale $1/\ve^2$
is needed. Namely, it was shown in \cite{Kha} and in a slightly more general situation in \cite{PK} and \cite{Bor}
that the slow motion $X^\ve(t),\, t\in[0,T]$ weakly converges as $\ve\to 0$ to a diffusion process.
If $B(x)=EB(x,\xi(s))$ is not zero but the system $\frac {dX(t)}{dt}=B(X(t))$ possesses an integral
of motion $H(X(t))\equiv const$ then \cite{CE} and \cite{BF} show that $H(X^\ve(t)),\, t\in[0,T]$
converges weakly as $\ve\to 0$ to a diffusion process. A more recent line of research
which relies substantially on the rough paths theory establishes weak convergence of
the slow motion $X^\ve$ to a diffusion when the process $\xi$ is generated by certain
 class of
dynamical systems, so that $\xi(t)=g\circ F^t$ where $g$ is a vector function and $F^t$
is a continuous or discrete time dynamical system with some hyperbolicity (see
\cite{CFKMZ} and references there).

Another, completely different, line of research dealt with extension of limit theorems from convergence
in distribution or weak convergence to strong approximations or strong invariance principles results.
This was done first in the one dimensional case using the Skorokhod embedding theorem in \cite{Str},
but since this approach does not work, in general, in the multidimensional case (see \cite{MP2}),
another method was developed in \cite{BP} to tackle the case of sums of weakly dependent random
vectors. This method is based on the Strassen--Dudley theorem which provides random variables having
given marginal distributions with the distance between the former estimated by means of the Prokhorov
distance between the latter which, in turn, is estimated through the difference between their
characteristic functions.
The, so-called, quantile transform method also appeared in the 1970s with its multidimensional
extension developed much later (see \cite{Zai} and references there) but it is applicable only to sums
of independent random vectors where it gives essentially optimal estimates for errors of
approximations. In all these papers sums of random vectors are approximated by a Brownian motion
considered on the same probability space with error estimates usually valid eventually almost surely,
i.e. when the number of summands tend to infinity. Another paper \cite{Ebe} dealt with strong
approximation of general stochastic processes by a similar to \cite{BP}, \cite{KP}, \cite{MP1} and
\cite{DP} method but it is not clear whether the conditions required there can be adapted to
our situation. Observe that error estimates are the crucial part of strong approximations
and not an almost sure vis-\' a-vis a weak convergence since by the Skorokhod representation theorem
(see, for instance, \cite{Bil}, p.70) it is always possible to realize a weak convergence as an almost
sure convergence on a sufficiently large (actually, huge) probability space.

All results mentioned above dealt with strong approximations when the limiting process is a Brownian
motion, and it seems, never a limiting process being a nontrivial multidimensional
diffusion was obtained before as a
result of strong approximations. In this paper we show that both the slow motion $X^\ve$ and a
corresponding diffusion $\Xi^\ve$ having the same initial condition
can be redefined on one sufficiently rich probability space so that
their uniform $L^{2M},\, M\geq 1$ distance on the time interval $[0,T]$ is bounded by $C(M)\ve^{\del}$
for some $C(M),\,\del>0$ and all $\ve>0$.
We note that the diffusion coefficients of $\Xi^\ve$ do not depend on $\ve$ while the Brownian motion
in its stochastic differential equation, in general, may depend on $\ve$, i.e. $\Xi^\ve$ remains the
same for all $\ve>0$ in the weak sense. Clearly, this result is substantially stronger than just the convergence of $X^\ve$ in distribution to a diffusion
which were obtained in previous papers cited above.
Observe also that the diffusion approximation obtained in \cite{Ki03} is not relevant here because it is obtained when the fast motion evolves on much shorter time intervals of order $1/\ve$ where the
slow motion is just a small diffusion perturbation of the averaged one and, essentially, we still
remain in the realm of Gaussian fluctuations. On the other hand, when (\ref{1.2}) is assumed and the
fast motion is considered on long time intervals of order $1/\ve^2$, we arrive at a true diffusion limit.

The time changed slow motion $Y^\ve(t)=X^\ve(\ve^2 t),\, 0\leq t\leq T/\ve^2$ satisfies the equation
\begin{equation}\label{1.3}
 \frac {dY^\ve(t)}{dt}=\ve B(Y^\ve(t),\xi(t))+\ve^2b(Y^\ve(t),\,\xi(t)).
\end{equation}
We consider also the corresponding discrete time setup given by the difference equation
\begin{equation}\label{1.4}
Y^\ve_d(n+1)=Y^\ve_d(n)+\ve B(Y^\ve_d(n),\,\xi(n))+\ve^2b(Y^\ve_d(n),\,\xi(n))
\end{equation}
where $n\in[0, T/\ve^2]$ is an integer and $d$ stands for "discrete". Returning back
to the original time scale we have
\begin{equation}\label{1.5}
X^\ve_d((n+1)\ve^2)=X^\ve_d(n\ve^2)+\ve B(X^\ve_d(n\ve^2),\,\xi(n))+\ve^2b(X^\ve_d(n\ve^2),\,\xi(n)).
\end{equation}
Considering the continuous time extension $X^\ve_d(t),\, t\in[0,T]$, either by the linear interpolation
between $X^\ve_d(k\ve^2)$ and $X^\ve_d((k+1)\ve^2)$ or taking $X^\ve_d(t)\equiv X^\ve_d(k\ve^2)$
when $k\ve^2\leq t<(k+1)\ve^2$, we prove that, again, if $\xi$ is a sufficiently fast mixing stationary
process then $X^\ve_d$ and a corresponding diffusion $\Xi^\ve$
can be redefined on the same sufficiently rich probability space so that the uniform $L^{2M},\,
M\geq 1$ distance between them on the time interval $[0,T]$ does not exceed $C(M)\ve^{\del}$ for
some $C(M),\,\del>0$. For instance, if we consider the particular case of (\ref{1.5}),
\begin{eqnarray}\label{1.6}
& X^{(1/\sqrt N)}_d((n+1)/N)=X^{(1/\sqrt N)}_d(n/N)+\frac 1{\sqrt N} \sig(X^{(1/\sqrt N)}_d(n/N))\xi(n)\\
& +\frac 1Nb( X^{(1/\sqrt N)}_d(n/N),\,\xi(n))\nonumber
 \end{eqnarray}
 where $\sig$ is a matrix function and $\xi(n),\, n=0,1,...$ is a sequence of independent identically
 distributed (i.i.d.) random vectors with zero mean and the identity covariance matrix then the
 diffusion approximation $\Xi$ of $X^{(1/\sqrt N)}_d$ as $N\to\infty$ will satisfy the stochastic
 differential equation
 \begin{equation}\label{1.7}
 d\Xi(t)=\sig(\Xi(t))dW(t)+b(\Xi(t))dt,
 \end{equation}
 where $W$ is the Brownian motion and the uniform $L^{2M},\, M\geq 1$ distance on $[0,T]$ between $X^{(1/\sqrt N)}_d$ and $\Xi$ can be estimated by $C(M)N^{-\del/2}$ for some $C(M),\,\del>0$ and all
 $\ve>0$ (here, again, for different $N$'s we may have to
 use different Brownian motions). We observe that if the random variables $\xi(n),\, n=0,1,...$ here
 are dependent then the above stochastic differential equation will have, in general, an extra drift term. These results enable us to use the above $X_d^{(1/\sqrt N)}$ for effective simulations and
 computations of diffusion processes since we can take in (\ref{1.6}) simple i.i.d. random vectors,
 say, those which have independent components taking on values $1$ or $-1$ with equal probability.
  A particular case of the difference equation (\ref{1.6}) was considered in \cite{He} to show the
 weak convergence of $X_d^{(1/\sqrt N)}$ to the diffusion $\Xi$ which, of course, could not provide
 any error estimates. In the one dimensional case of this particular setup it was still possible
 to use an extended version of
  the Skorokhod embedding (into martingales) theorem to produce discrete approximations of diffusions with estimates of errors (see \cite{BDG}).

  In the last section of this paper we use our discrete time approximations of diffusions for
  computation of values of Dynkin's optimal stopping games with payoffs being functionals on paths
   of a diffusion process.
  Of course, error estimates of such computations cannot be obtained relying on weak convergence
   results as in \cite{Dol} and our strong approximations estimates become necessary here.
   It is well known that the
    value of discrete time Dynkin's games can be obtained by the dynamical programming (backward recursion) procedure while it is difficult to compute value of a continuous time Dynkin game directly. We consider path dependent
    payoffs, so, except for few specific cases, it is impossible, in general,
     to compute values of such games
    using free boundary partial differential equations. Observe that the standard time discretization of a diffusion does not help much in the above dynamical programming procedure
    since it involves computation of conditional expectations with respect to large $\sigma$-algebras,
    and so the possibility to choose finitely many simple vectors as possible values of $\xi(n)$'s in
    (\ref{1.6}), which would require computation of conditional expectations with respect to simple finite $\sig$-algebras, become useful and may provide
    a better computational tool than the well known Euler--Maruyama approximation
    of solutions of stochastic differential equations. This yields also an application to mathematical finance
     enabling us to compute effectively prices of game (and also of European and American) options
     in markets where the underlying stock price evolves according to a general diffusion process and not just as a geometric Brownian motion (for more
     details see \cite{Ki21}).

     The structure of this paper is the following. In the next section we formulate precisely our
     main results. In Sections \ref{sec3} and \ref{sec4} we prove our main approximation result in the
     continuous time case while the discrete time case is treated in Section \ref{sec5}. In Section \ref{sec6} we deal with  Dynkin's games.

\section{Preliminaries and main results}\label{sec2}\setcounter{equation}{0}

We start with a complete probability space $(\Om,\,\cF,\, P)$, a stationary process
$\xi(t),\,-\infty<t<\infty$ and a family of countably generated $\sig$-algebras $\cF_{st},\,-\infty
\leq s\leq t\leq\infty$ completed by sets of zero probability and
such that $\xi(t)$ is $\cF_{tt}$-measurable for any $t\in(-\infty,\infty)$ and
$\cF_{st}\subset\cF_{s't'}\subset\cF$ if $s'\leq s\leq t\leq t'$ where
$\cF_{s,\infty}=\cup_{t:t\geq s}\cF_{st}$ and $\cF_{-\infty,t}=\cup_{s:s\leq t}\cF_{st}$. Recall, that such two parameter families of $\sig$-algebras
serve as a standard tool in the study of limit theorems for sums (or integrals) of families of weakly dependent random variables (or vectors)
where we have to describe dependencies between different parts of these sums
(see, for instance, \cite{Bil}, \cite{Bra}, \cite{Kha}, \cite{KP}, \cite{KV},
\cite{MP1} etc.)

We will measure the dependence between $\sig$-algebras $\cG$ and $\cH$ by the $\phi$-coefficient defined by
\begin{eqnarray}\label{2.0}
&\phi(\cG,\cH))=\sup\{\vert\frac {P(\Gam\cap\Del)}{P(\Gam)}-P(\Del)\vert:\, P(\Gam)\ne
0,\,\Gam\in\cG,\,\Del\in\cH\}\\
&=\frac 12\sup\{\| E(g|\cG)-Eg\|_\infty:\, g\,\,\mbox{is $\cH$-measurable and }\,\|g\|_\infty=1\}\nonumber
\end{eqnarray}
(see \cite{Bra}) where $\|\cdot\|_\infty$ is the $L^\infty$-norm. For each $u\geq 0$ we set also
\begin{equation}\label{2.1}
\phi(u)=\sup_t\phi(\cF_{-\infty,t},\cF_{t+u,\infty}).
\end{equation}
If $\phi(u)\to 0$ as $u\to\infty$ then the probability measure $P$ is called $\phi$-mixing with
respect to the family $\{\cF_{st}\}$. We assume that
\begin{equation}\label{2.2}
D=\sup_{u\geq 0}(\phi(u)(u^{2M}+u^4))<\infty
\end{equation}
where $M\geq 1$ is an integer.

We will deal with the systems of ordinary differential equations (\ref{1.1}) containing a small parameter $\ve>0$ and will assume that the coefficients $B$ and $b$ in  (\ref{1.1})
are maps $B,b:\,\bbR^d\times\bbR^\nu\to\bbR$ such that $B$ is twice and $b$ is
once differentiable in the first variable, they are Borel measurable in the second variable
 and satisfy uniform bounds
\begin{equation}\label{2.3}
\max\big(| B(x,\xi)|,\,|\nabla_x B(x,\xi)|,\,
| \nabla^2_xB(x,\xi)|,\,| b(x,\xi)|,\,|\nabla_x b(x,\xi)|\big)\leq L
\end{equation}
for some constant $L\geq 1$, where $B=(B_1,...,B_d)$ and $b=(b_1,...,b_d)$ are $d$-dimensional
vectors and we take the Euclidean norms
\begin{eqnarray*}
&|B(x,\xi)|=(\sum_{i=1}^dB_i^2(x,\xi))^{1/2},\,
|b(x,\xi)|=(\sum_{i=1}^db_i^2(x,\xi))^{1/2},\\
&\,|\nabla_xB(x,\xi)|=(\sum_{i,j=1}^d|\frac {\partial B_i(x,\xi)}{\partial x_j}|^2)^{1/2}, |\nabla_xb(x,\xi)|=(\sum_{i,j=1}^d|\frac {\partial b_i(x,\xi)}{\partial x_j}|^2)^{1/2},\\
& |\nabla^2_xB(x,\xi)|=(\sum_{i,j,k=1}^d|\frac {\partial^2 B_i(x,\xi)}{\partial x_j
\partial x_k}|^2)^{1/2}.
\end{eqnarray*}
To make the exposition more readable we will provide a detailed proof under these uniform
 boundedness conditions and in Remark \ref{rem2.5} below we formulate weaker moment conditions under which our proofs still can go through.

The vector valued stationary process $\xi(t,\om)$ is supposed to be progressively measurable
(see, for instance, Section 7.2.2 in \cite{Ki20}) with respect to the filtration $\{\cF_{-\infty,t},\,
t\in(-\infty,\infty)\}$, and so $\int_0^tB(x,\,\xi(s,\om))ds$ as a process in $t$ is also
progressively measurable with respect to the same filtration, in particular, the latter integral
is $\cF_{-\infty,t}$-measurable.
Since we do not assume continuity of $B(x,y)$ in $y$ and $\xi(t)$ in $t$, we understand the equations
(\ref{1.1}) and (\ref{1.3}) in the integral form
\[
X^\ve(t)=X^\ve(0)+\int_0^t\big(\frac 1\ve B(X^\ve(s),\xi(s/\ve^2))+b(X^\ve(s),\xi(s/\ve^2))\big)ds,
\,\, t\in[0,T]
\]
and
\[
 Y^\ve(t)=Y^\ve(0)+\ve\int_0^t\big(B(Y^\ve(s),\xi(s))+\ve b(Y^\ve(s),\,\xi(s))\big)ds,\,\, t\in[0,T/\ve^2]
 \]
 which comes back to the differential form (\ref{1.1}) and (\ref{1.3}) only for Lebesgue almost
 all $t$. Since the solutions $X^\ve$ and $Y^\ve$ of this integral equations can be obtained by
 the Picard successive approximations method, it is easy to see (inductively and passing to
 the limit) that the processes $X^\ve(\ve^2t,\om)$
 and $Y^\ve(t,\om)$ are also progressively measurable with respect to the filtration
 $\{\cF_{-\infty,t},\, t\in(-\infty,\infty)\}$ and the same is true for pairs
  $X^\ve(\ve^2t,\om),\,\xi(t,\om)$ and $Y^\ve(t,\om),\,\xi(t,\om)$. Hence,
  $\int_0^tB(X^\ve(\ve^2s,\om),\,\xi(s,\om))ds$ as a process in $t$ is progressively measurable,
  as well, and, in particular, it is adapted with respect to the above filtration, i.e. the latter
  integral is $\cF_{-\infty,t}$-measurable.
  In general, we can assume that the stationary process $\xi(t),\,-\infty<t<\infty$
 takes values in a Polish space but since all such spaces are isomorphic to a subset of the real line $\bbR^1$, we can assume that $\xi(t)$ is real or vector valued though this does not matter for
 our method.
 Finally, we assume that for any $x\in\bbR^d$ (and any $-\infty<s<\infty$ by stationarity of $\xi$)
the equality (\ref{1.2}) holds true. In Remark \ref{rem2.6} we will discuss an extension where
(\ref{1.2}) is replaced by the assumption that the averaged system $\frac {X(t)}{dt}=B(X(t))$, where
$B(x)=EB(x,\xi(s))$, possesses an integral of motion (conservation law) as in \cite{CE} and \cite{BF}.

Set
\[
c(x,u,v)=E(\nabla_xB(x,\xi(u))B(x,\xi(v)))
\]
where $\nabla_xB(x,y)B(x,z)$ is the vector with the components
\[
(\nabla_xB(x,y)B(x,z))_i=\sum_{j=1}^d\frac {\partial B_i(x,y)}{\partial x_j}B_j(x,z).
\]
Define also
\[
a_{jk}(x,u,v)=E(B_j(x,\xi(u))B_k(x,\xi(v))).
\]
It will be shown in the next section that the limits
\begin{equation}\label{2.6}
c(x)=\lim_{t\to\infty}\frac 1t\int_s^{s+t}du\int_{s-t}^uc(x,u,v)dv=\lim_{t\to\infty}\frac 1t\int_0^t
du\int_{-t}^uc(x,u,v)dv
\end{equation}
and, for $j,k=1,...,d,$
\begin{equation}\label{2.7}
a_{jk}(x)=\lim_{t\to\infty}\frac 1t\int_s^{s+t}\int_{s}^{s+t}a_{jk}(x,u,v)dudv=\lim_{t\to\infty}
\frac 1t\int_0^t\int_0^ta_{jk}(x,u,v)dudv
\end{equation}
exist uniformly in $s$.

We will see that under our conditions the matrix $A(x)=(a_{jk})$ is symmetric and twice differentiable
in $x$, and so it has a symmetric Lipschitz continuous in $x$ square root $\sig(x)$, i.e. we have the
representation (see \cite{Fre} and Sections 5.2 and 5.3 in \cite{SV}),
\begin{equation}\label{2.8}
A(x)=\sig^2(x),
\end{equation}
and both the uniform bound of the norm and the Lipschitz constant of $\sig$ will be denoted again
 by $L$. In fact, for our purposes it suffices to have the representation $A(x)=\sig(x)\sig^*(x)$ with
a Lipschitz continuous matrix $\sig$ where $\sig^*$ is the conjugate to $\sig$.
It turns out that both $b(x)=Eb(x,\xi(s))$ and $c(x)$ given by (\ref{2.6}) are Lipschitz continuous,
as well. Thus, there exists a unique solution $\Xi$ of the stochastic differential equation
\begin{equation}\label{2.9}
d\Xi(t)=\sig(\Xi(t))dW(t)+(b(\Xi(t))+c(\Xi(t)))dt
\end{equation}
where $W$ is the standard $d$-dimensional Brownian motion. When a non negatively definite symmetric
 matrix $A(x)$ is fixed then any solution of (\ref{2.9}) with any matrix $\sig$ satisfying
  $A(x)=\sig(x)\sig^*(x)$ has the same path distribution since this leads to the same Kolmogorov
  equation and to the same martingale problem (see \cite{SV}).

\begin{theorem}\label{thm2.1} Suppose that the conditions (\ref{1.2}), (\ref{2.2}) and
(\ref{2.3}) hold true and that a symmetric Lipschitz continuous matrix $\sig(x)$ satisfying
 (\ref{2.8}) is fixed. Then the slow motion $X^\ve$ and the diffusion $\Xi=\Xi^\ve$ having the
  same initial condition $X^\ve(0)=\Xi^\ve(0)=x_0$
 can be redefined preserving their distributions on the same sufficiently rich probability space, which contains also an i.i.d.
sequence of uniformly distributed random variables, so that for any integer $M\geq 1$ satisfying
 (\ref{2.2}) and all positive $\ve\leq\ve_0=(2\cdot 10^{640d}d^{80d})^{-5/2}$,
\begin{equation}\label{2.10}
E\sup_{0\leq t\leq T}|X^\ve(t)-\Xi^\ve(t)|^{2M}\leq C_0(M)\ve^{\del },
\end{equation}
 where we can take $\del=\frac 1{500d}$ and
$C_0(M)=3^{2M}(C_9(M)\exp(C_{10}(M)T)+92L)^{2M}+2^{4M}+2^{6M}M^{3M}L^{2M}T^{1/2})$
 with $C_9(M)$ and $C_{10}(M)$ defined in Section \ref{subsec4.4}. Here $\Xi$ depends on $\ve$ in the strong but not in the weak sense, i.e. the coefficients in (\ref{2.9}) do not depend on $\ve$ but for each $\ve>0$ in order to satisfy (\ref{2.10}) we may have to choose an appropriate Brownian motion $W=W_\ve$.
  In particular, the Prokhorov distance between the distributions of $X^\ve$ and of $\Xi$ is bounded by $(C_0(2)\ve^{\del })^{1/3}$.
\end{theorem}

Clearly, the estimate (\ref{2.10}) is meaningful only for small $\ve$ and we provide it for all
$\ve\in(0,\ve_0]$ while, of course, an explicit estimate for $\ve>\ve_0$ can also be obtained
in (\ref{2.10}) just by estimating $|X^\ve(t)|\leq TL(\ve_0^{-1}+1)$ and by using the standard
martingale estimates for stochastic integrals in order to bound  $E\sup_{0\leq t\leq T}|\Xi(t)|^{2M}$
by $2^{M-1}T^ML^{2M}((4M^3)^M(2M-1)^{-M})+T^M)$.
Observe also that for any positive integer $p\leq 2M$ we can obtain $E\sup_{0\leq t\leq
 T}|X^\ve(t)-\Xi(t)|^p\leq(C_0(M)^{p/2M}\ve^{\del p/2M}$ from (\ref{2.10}) just by applying the
  Jensen (or H\" older) inequality
$E|Z|^p\leq (EZ^{2M})^{p/2M}$. We do not attempt to optimize constants in our estimates since
even for sums of weakly dependent (multidimensional) random vectors, the currently known
 applicable methods yielding strong approximations yield estimates which seem to be far from
optimal. On the other hand, we provide explicitly all constants, so that our estimates may have also practical interest.
 The key idea in the proof of Theorem \ref{thm2.1} is to freeze the slow motion
 at certain times $\ve^2t_{k-1}$ and then to make (conditional) strong approximations of integrals $\ve\int_{t_k}^{t_{k+1}}B(X^\ve(\ve^2t_{k-1}),\xi(u))du$ by Gaussian processes with covariance matrices $A(X^\ve(\ve^2t_{k-1}))$ viewing these as integrals of weakly
 dependent random vectors and employing, essentially, the technique from \cite{KP},
  gluing them together and approximating the resulting process by the true diffusion.

 Next, we will describe the discrete time version of the above result. We consider now the difference equations (\ref{1.5}) and assume that the coefficients $B$ and $b$ there satisfy the conditions
 (\ref{2.3}). The setup includes again a family of $\sig$-algebras $\cF_{st}\subset\cF,\,-\infty
 \leq s\leq t\leq\infty$ with the same properties as above but now $s$
 and $t$ take on only integer values. The $\phi$-dependence coefficient is defined again by (\ref{2.1}) only $t$ there runs along integers. The definition of the coefficients $c(x)$ and $a_{jk}(x)$ are now given by
 \begin{equation}\label{2.12}
 c(x)=\lim_{n\to\infty}\frac 1n\sum_{l=\im}^{\im+n}\sum_{m=\im-n}^{l-1}c(x,l,m)
 \end{equation}
 and
 \begin{equation}\label{2.13}
 a_{jk}(x)=\lim_{n\to\infty}\frac 1n\sum_{l,m=\im}^{\im+n}a_{jk}(x,l,m)
 \end{equation}
 where the definitions of $c(x,l,m)$ and of $a_{jk}(x,l,m)$ are the same as in the continuous time case
 and the existence of the limits (\ref{2.12}) and (\ref{2.13}) will be proved in Section \ref{sec5}.
  Again, we will
 see that the matrix $A(x)=(a_{jk}))$ is twice differentiable in $x$, and so by \cite{Fre} there exists
 a symmetric Lipschitz continuous matrix function $\sig(x)$ satisfying (\ref{2.8}). This enables us to
 define again the diffusion $\Xi$ as a solution of the stochastic differential equation (\ref{2.9}).

 Next, we extend $X^\ve_d$ to the continuous time setting
 \begin{equation}\label{2.14}
 X^\ve_d(t)=X^\ve_d(n\ve^2)\quad\mbox{if}\,\, n\ve^2\leq t<(n+1)\ve^2,\,t\in[0,T].
 \end{equation}
 The interpolation definition
 \begin{equation}\label{2.15}
 X^\ve_d(t)=(t-\ve^2n)X^\ve_d((n+1)\ve^2)+(\ve^2(n+1)-t))X^\ve_d(n\ve^2)
 \end{equation}
 leads to the same results but we will use (\ref{2.14}). The discrete time version of Theorem  \ref{thm2.1} is the following result.

 \begin{theorem}\label{thm2.2}
 Suppose that the conditions (\ref{1.2}), (\ref{2.2}) and (\ref{2.3}) hold true
 and that a symmetric Lipschitz continuous matrix $\sig(x)$ satisfying (\ref{2.8}) is fixed.
 Then $X^\ve_d$ and the diffusion $\Xi=\Xi^\ve$ having the same initial condition $X^\ve(0)=\Xi^\ve(0)=x_0$
  can be redefined without changing their distributions
 on the same sufficiently rich probability space, which contains also an i.i.d. sequence of uniformly
 distributed random variables, so that for any integer $M\geq 1$ satisfying (\ref{2.2}) and all
 positive $\ve\leq\ve_0$,
 \begin{equation}\label{2.16}
 E\sup_{0\leq t\leq T}|X^\ve_d(t)-\Xi^\ve(t)|^{2M}\leq C_0(M)\ve^{\delta}
 \end{equation}
  where $C_0(M)$ and $\del>0$ can be taken the same as in Theorem \ref{thm2.1}
 and the dependence of $\Xi$ on $\ve$ is as described there. Again, the Prokhorov distance between
 the distributions of $X^\ve$ and $\Xi$ is bounded by $(C_0(2)\ve^{\del })^{1/3}$.
 \end{theorem}

  We can view (\ref{1.5}) also as a convenient form of approximation of a given diffusion process
  $\Xi$ which, say, solves the stochastic differential equation (\ref{2.9}) with $c(x)=0$. To do this
   we consider the difference equations
   \begin{eqnarray}\label{2.18}
   &X^{(1/\sqrt N)}_d((n+1)/N)\\
   &=X^{(1/\sqrt N)}_d(n/N)+\frac 1{\sqrt N}\sig(X^{(1/\sqrt N)}_d(n/N))\xi(n)+
   \frac 1Nb(X^{(1/\sqrt N)}_d(n/N)),\nonumber
   \end{eqnarray}
  $n=0,1,...,N-1$, where we can take $\xi(n)=(\xi_1(n),...,\xi_d(n)),\, n=0,1,...$ to be an i.i.d.
   sequence of random vectors with $E\xi(0)=0$ and $E(\xi_i(k)\xi_j(l))=\del_{ij}\del_{kl}$ where
   $\del_{mn}$ is the Kronecker delta.  In this case $c(x,l,m)=0$ if $l\ne m$ and by (\ref{2.12}) we see that $c(x)\equiv 0$. Now, assuming that the $d\times d$ matrix $\sig(x)$ is twice differentiable
    and the vector $b(x)$ is once differentiable we will obtain according to Theorem \ref{thm2.2} an
   approximation of $\Xi$ with the $L^{2M}$-precision of $C_0(M)N^{-\del/2}$. To make random vectors $\xi(n)$ simplest possible we can take them with independent components taking on values $1/d$ and $-1/d$ with probability $1/2$.

   Next, we will describe an application of our results to computations of values of Dynkin's optimal
   stopping games with the payoff function having the form
  \begin{equation}\label{2.19}
  R^\Xi(s,t)=G_s(\Xi)\bbI_{s<t}+F_t(\Xi)\bbI_{t\leq s}
  \end{equation}
  where $\Xi$ is a diffusion solving the stochastic differential equation
  \begin{equation}\label{2.20}
  d\Xi(t)=\sig(\Xi(t))dW(t)+b(\Xi(t))dt,\quad t\in[0,T],\,\, \Xi(0)=x_0.
  \end{equation}
  Here, $G_t\geq F_t$ and both are functionals on paths for the time interval $[0,t]$ satisfying
  certain regularity conditions specified below. Thus, if the first player stops at the time $s$ and
  the second one at the time $t$ then the former pays to the latter the amount $R^\Xi(s,t)$. The game
  runs until a termination time $T<\infty$ when the game stops automatically, if it was not stopped
  before, and then the first player pays to the second one the amount $G_T(\Xi)=F_T(\Xi)$. Clearly,
  the first player tries to minimize the payment while the second one tries to maximize it. Under the
  conditions below this game has the value (see, for instance, Section 6.2.2 in \cite{Ki20}),
  \begin{equation}\label{2.21}
  V^{\Xi}=\inf_{\sig\in\cT_{0T}^\Xi}\sup_{\tau\in\cT_{0T}^{\Xi}}ER^{\Xi}(\sig,\tau)
  \end{equation}
  where $\cT^\Xi_{0T}$ is the set of all stopping times $0\leq\tau\leq T$  with respect to the
  filtration $\cF_t^\Xi,\, t\geq 0$ generated by the diffusion $\Xi$ or, which is the same, generated
  by the Brownian motion $W$.

  We assume that $F_t$ and $G_t,\, t\in[0,T]$ are continuous functionals on the space $M_d[0,t]$ of
  bounded Borel measurable maps from $[0,t]$ to $\bbR^d$  considered with the uniform metric
   $d_{0t}(\up,\tilde\up)=\sup_{0\leq s\leq t}|\up_s-\tilde\up_s|$ and there exists a constant $K>0$ such that
   \begin{equation}\label{2.22}
   |F_t(\up)-F_t(\tilde\up)|+|G_t(\up)-G_t(\tilde\up)|\leq Kd_{0t}(\up,\tilde\up)
   \end{equation}
   and
   \begin{equation}\label{2.23}
   |F_t(\up)-F_s(\up)|+|G_t(\up)-G_s(\up)|\leq K(|t-s|(1+\sup_{u\in[s,t]}|\up_u|)+\sup_{u\in[s,t]}
   |\up_u-\up_s|).
   \end{equation}

 Next, we will consider Dynkin's games with payoffs based on the discrete time slow motion $X^\ve_d$
 obtained by the difference equations
 \begin{equation}\label{2.26}
   X^{\ve}_d((n+1)\ve^2)=X^{\ve}_d(n\ve^2)+\ve\sig(X^{\ve}_d(n\ve^2))\xi(n)+\ve^2b(X^{\ve}_d(n\ve^2)),
   \,\, X^\ve_d(0)=x_0.
 \end{equation}
 where $\xi(n)=(\xi_1(n),...,\xi_d(n)),\,-\infty<n<\infty$ is a stationary $\phi$-mixing
 sequence of bounded random
 vectors such that $E\xi(0)=0$, $E|\xi_l(0)|^2=1$ and $E(\xi_i(m),\xi_j(n))=\del_{ij}\del_{mn}$
 for all $i,j=1,...,d$ and any integers $m,n$. This ensures that $A(x)=\sig(x)\sig^*(x)$ and $c(x)\equiv 0$ in (\ref{2.12}) and
 assuming that the matrix $\sig$ is twice and the vector $b$ is once differentiable and they are
 bounded, i.e. the condition (\ref{2.3}) for $B(x,\xi(n))=\sig(x)\xi(n)$ and $b(x,\xi(n))=b(x)$ hold
 true we see that $X^\ve_d$ approximates the diffusion $\Xi$ in the sense of Theorem \ref{thm2.2}
 provided that (\ref{2.2}) is satisfied. We observe that the simplified form of $B$ and $b$ is not
 important for our method though the main motivation for the result below is to approximate the
 game value of the continuous time Dynkin game by a simpler discrete time model, and so from this point of view the most general setup for the latter does not bring additional value.

 We extend, again, $X^\ve_d$ to the continuous time in the piece-wise constant fashion (\ref{2.14})
 and define the payoff based on $X^\ve_d$ of the corresponding Dynkin game by
 \begin{equation}\label{2.27}
 R^\ve(s,t)=G_s(X^\ve_d)\bbI_{s<t}+F_t(X^\ve_d)\bbI_{t\leq s}.
 \end{equation}
 Let $\cF_{mn},\, m\leq n$ be the $\sig$-algebra generated by $\xi(m),...,\xi(n)$ and $\cT_{mn}$ be
 the set of all stopping times with respect to the filtration $\cF_{-\infty,k},\, k\geq 0$ taking on values
 $m,m+1,...,n$. We allow also any stopping time to take on the value $\infty$, i.e. we allow players
 not to stop the game at all, but anyway the game is stopped automatically at the termination time
 $T<\infty$ and then the first player pays to the second one the amount $G_T(X^\ve_d)=F_T(X^\ve_d)$.
 Set $N_\ve=[T/\ve^2]$ (where $[\cdot]$ stands for the integral part)
 then the game value of the Dynkin game in this setup is given by
 \begin{equation}\label{2.28}
 V^\ve=\inf_{\zeta\in\cT_{0N_\ve}}\sup_{\eta\in\cT_{0N_\ve}}ER^\ve(\ve^2\zeta,\ve^2\eta).
 \end{equation}

\begin{theorem}\label{thm2.3} Set $B(x,\xi(n))=\sig(x)\xi(n)$ and assume that the stationary process
$\xi(n)$ satisfies both the conditions above and the conditions of Theorem
 \ref{thm2.2} for such $B$. Suppose that conditions (\ref{2.22}) and (\ref{2.23}) hold true, as well.
  Then for any positive $\ve\leq\ve_0$,
 \begin{equation}\label{2.31}
 |V^\Xi-V^\ve|\leq\tilde C\ve^{\delta/2}
 \end{equation}
 where $V^\Xi$ and $V^\ve$ are given by (\ref{2.21}) and (\ref{2.28}), respectively,
 $\del>0$ is the same as in Theorem \ref{thm2.2}, $\tilde C$ can be estimated explicitly from
 Lemmas \ref{lem4.6}, \ref{lem6.1}--\ref{lem6.4} and the inequalities (\ref{6.26})--(\ref{6.29}).
 \end{theorem}

 We observe that the main advantage in computation $V^\ve$ in comparison to $V^\Xi$ is the possibility
 to use the dynamical programming (backward recursion) algorithm. Namely, set
 $V^\ve_{N_\ve}=F_{\ve^2N_\ve}(X^\ve_d)$ and recursively for $n=N_\ve-1,...,1,0$,
 \begin{equation}\label{2.32}
 V^\ve_n=\min\big(G_{\ve^2n}(X^\ve_d),\,\max(F_{\ve^2n}(X^\ve_d),\, E(V^\ve_{n+1}|\cF_{-\infty,n}))\big).
 \end{equation}
 Then $V_0^\ve=V^\ve$ (see, for instance, Section 6.2.2 in \cite{Ki20}). Of course, the computation of conditional expectations above becomes complicated if the $\sig$-algebras $\cF_{-\infty,n }$ are big but if we
  choose simple independent random vectors $\xi(n)$ in (\ref{2.18}), as explained there, then these $\sig$-algebras contain not so many sets and the conditional expectations can be computed easily (see \cite{Ki21} for more discussion).
  Observe also that in the particular case when the diffusion $\Xi$ is just a multidimensional Brownian
   motion, a result similar to Theorem \ref{thm2.3} was obtained in \cite{Ki07} where it was
   sufficient to consider the standard normalized sums of random vectors $\xi(n)$ rather than the
   more subtle case of difference equations (\ref{1.5}).

   In Theorem \ref{thm2.3} we will rely on the specific construction of the diffusion $\Xi$ which will
    be obtained in the proof of Theorem \ref{thm2.2} using the strong approximation theorem exhibited
    in Section \ref{sec4}. Namely, the strong approximation (\ref{2.16}) does not lead directly to the
    estimate (\ref{2.31}) because the sets of stopping times in the definitions (\ref{2.21}) and
     (\ref{2.28}) are different since they depend on filtrations of $\sig$-algebras with respect to which they  are considered. Moreover, in order that Theorem \ref{thm2.3} will make sense we
   have to be sure that the game value $V^\Xi$ does not depend on a path-wise representation of the
   diffusion $\Xi$, i.e. that $V^\Xi$ will be the same for any (weak) solution of (\ref{2.20}) no
   matter which Brownian motion we choose. This follows from \cite{Dol} (see p.p. 1893--1894 there),
   namely it turns out that once we choose a time continuous version of the diffusion $\Xi$, the value
   $V^\Xi$ depends only on the distribution of $\Xi$ on the space of its continuous paths. In other
   words, for any time continuous diffusion $\tilde\Xi$ with the drift $b$ and a diffusion matrix
   $\tilde\sig(x)$ satisfying $\tilde\sig(x)\tilde\sig^*(x)=\sig(x)\sig^*(x)$ the value $V^{\tilde\Xi}$
    of the game with payoffs built on $\tilde\Xi$ in place of $\Xi$ will be equal $V^\Xi$. Observe also
    that taking a very large $G$ so that the first player will never stop the game we will reduce the result of Theorem \ref{thm2.3} to the standard one person optimal stopping setup (in which case
    Theorem \ref{thm2.3} also seems to be new). This also can be proved directly repeating and slightly simplifying the arguments in the proof of Theorem \ref{thm2.3}.

    \begin{remark}\label{rem2.4}
 Having in mind applications to financial mathematics it is useful to have the approximation estimates of Theorem \ref{thm2.3} under more general than
 (\ref{2.22}) and (\ref{2.23}) conditions which include also exponential
 functionals and allow to represent a stock evolution by an exponential of
 a diffusion. Assume in place of (2.19) and (2.20) that
 \begin{eqnarray*}
  & |F_t(\up)-F_t(\tilde\up)|+|G_t(\up)-G_t(\tilde\up)|\\
  &\leq K(d_{0t}(\up,\tilde\up)+\bbI_{\sup_{0\leq
    u\leq t}|\up_u-\tilde\up_s|>1})\exp(K\sup_{0\leq u\leq t}(|\up_u|+|\tilde\up_u|))\nonumber
   \end{eqnarray*}
   and
   \begin{equation*}
   |F_t(\up)-F_s(\up)|+|G_t(\up)-G_s(\up)|\leq K(|t-s|+\sup_{u\in[s,t]}|\up_u-\up_s|)\exp(K\sup_{0\leq
    u\leq t}|\up_u|).
   \end{equation*}
The above assumptions allow payoff functionals described as follows.
For each $\up=(\up^{(1)},...,\up^{(d)})\in M_d[0,t]$ denote by
$ex(\up)\in M_d[0,t]$ the map from $[0,t]$ to $\bbR^d$ defined by
$ex(\up)_s=(ex(\up)_s^{(1)},...,ex(\up)_s^{(d)})$ where $ex(\up)_s^{(i)}=
e^{\up_s^{(i)}}$ for each $i=1,...,d$ and $s\in[0,t]$.
Then the functionals having the form $F_t(\up)=\tilde F_t(ex(\up))$ and
$G_t(\up)=\tilde G_t(ex(\up))$ satisfy the above conditions provided
$\tilde F$ and $\tilde G$ satisfy (\ref{2.22}) and (\ref{2.23})
(in place of $F$
and $G$ there). Now let $\Xi=(\Xi_1,...,\Xi_d)$ be a diffusion with bounded
coefficients solving the stochastic differential equation (\ref{2.20}).
Then by
the It\^ o formula $ex(\Psi)_s=(ex(\Psi)_s^{(1)},...,ex(\Psi)_s^{(d)})=
(e^{\Xi_1(s)},...,e^{\Xi_d(s)})$ has the stochastic differential
$dex(\Psi)_s=(dex(\Psi)_s^{(1)},...,dex(\Psi)_s^{(d)})$ with
\[
dex(\Psi)_s^{(i)}=ex(\Psi)_s^{i}(\sum_{j=1}^d\sig_{ij}(\Xi(s))dW_j(s)+
(b_i(\Xi(s))+\frac 12)dt).
\]
Thus, if $\sig$ and $b$ are constant matrix and vector function, respectively,
$ex(\Psi)$ has the form of a multidimensional geometric Brownian motion
(see, for instance, \cite{Ki20}).
Hence, obtaining an approximation of the value of a Dynkin game with
payoffs $F_t(\Xi)$ and $G_t(\Xi)$ is the same as approximation of the price
of a game option with payoffs $\tilde F(ex(\Xi))$ and $\tilde G(ex(\Xi))$
in the multi asset Black-Scholes financial market with stocks evolution
described by $ex(\Xi)$.

   For detailed proofs under these conditions (when $\xi(n)$'s form a sequence of independent random
   vectors) we refer the reader to \cite{Ki21} and here we
   will only indicate how to modify arguments in Section \ref{sec6} to incorporate this setup, as well. Namely,
  carrying over our proof in Section \ref{sec6} under these more general conditions will affect our estimates (\ref{6.5}), (\ref{6.8}), (\ref{6.15}), (\ref{6.21}), (\ref{6.27}) and (\ref{6.29}) there.
  All expressions which we have to estimate there will have now the form
  \[
  E((|\Te|+\bbI_{|\Te|>1})e^{|\Up|})\leq((E|\Te|^2)^{1/2}+(P\{|\Te|>1\})^{1/2})(Ee^{2|\Up|})^{1/2}
  \]
  where the estimates for the first factor related to $\Te$ are obtained in Section \ref{sec6} and they can be used directly. On the other hand, the second factor requires
   additional estimates as in Lemma 5.1 from \cite{Ki21} with
  $\Up$ above having the form $\Up=\sup_{o\leq t\leq T}|M^\ve(t)|+Q_\ve$ where $Q_\ve$ is uniformly bounded and
   \[
   \mbox{either}\,\, M^\ve(t)=\ve K\sum_{0\leq n\leq t/\ve^2}\sig(X^\ve_d(n\ve^2))\xi(n)\,\,\mbox{or}
   \,\, M^\ve(t)=K\int_0^t\sig(\Xi(s))dW(s).
   \]
   In the second case the expectation of the exponent of a stochastic integral can be estimated directly using, for instance, exponential martingales (see, for instance, \cite{Ki20},
   Section 7.4.2). Since $e^{|a|}\leq e^a+e^{-a}$ we have to estimate only $\sup_{0\leq t\leq T}e^{2M^\ve(t)}$ and $\sup_{0\leq t\leq T}e^{-2M^\ve(t)}$. When $M^\ve(t)$ equals the first
   expression above the required estimates were carried out in \cite{Ki21} when $\xi(n)$'s are
   independent. Observe that for effective computations of values of game options with log-diffusion
   stock price evolution it is natural to choose a simplest possible diffusion approximation scheme,
   and so we do not need to look beyond independent sequences of $\xi(n)$'s since they already do
   the job.
   \end{remark}

    \begin{remark}\label{rem2.5} Theorems \ref{thm2.1}--\ref{thm2.3} can be obtained assuming moment rather than uniform bounds, namely, in place of (\ref{2.3}) requiring that for some $m$ big enough,
   \begin{eqnarray*}
&E\sup_{x\in\bbR^d}\max\big(| B(x,\xi(0))|^m,\,|\nabla_x B(x,\xi(0))|^m,\,
| \nabla^2_xB(x,\xi(0))|^m,\\
&| b(x,\xi(0))|^m,\,|\nabla_x b(x,\xi(0))|^m\big)\leq L<\infty.\nonumber
\end{eqnarray*}
It is possible also to replace the $\phi$-mixing coefficient by more general dependence coefficients
between pairs of $\sig$-algebras $\cG,\cH\subset\cF$ defined by
\[
\varpi_{q,p}(\cG,\cH)=\sup\{\|E(g|\cG)-Eg\|_p:\, g\,\,\mbox{is}\,\,\cH-\mbox{measurable and}\,\,
\|g\|_q\leq 1\}.
\]
The proofs proceed then essentially in the same way supplementing them by the frequent use of the
H\" older inequality (cf. \cite{KV}). Of course, under these conditions the numbers $M$, for
which (\ref{2.10}) and
(\ref{2.16}) will hold true, will depend on $m$ and on assumptions concerning $\varpi_{q,p}$.
\end{remark}
\begin{remark}\label{rem2.6} Theorem \ref{thm2.1} can be extended to the case when the condition
(\ref{1.2}) is replaced by the assumption that the averaged system
\[
\frac {dX(t)}{dt}=B(X(t)),\,\, B(x)=EB(x,\xi(0))
\]
has integrals of motion $H(x)=(H_1(x),...,H_l(x)),\, x\in\bbR^d$, i.e. $H_i(X(t))=H_i(x)$,
$x=X(0)$ for all $t\geq 0$ and $i=1,...,l$. In the particular case $d=2$ let
 $H(x),\, x\in\bbR^2$ be a bounded integral of motion which is supposed to be
 trice differentiable with uniformly bounded derivatives. Moreover, the level sets $C(y)=
 \{ x:\, H(x)=y\}$ are supposed to be closed connected curves without intersections.
 Now, instead of obtaining a diffusion
 approximation for the process $X^\ve$ under the condition (\ref{1.2}) we do this for the process
 $Y^\ve(t)=H(X^\ve(t))$. Now, in place of Lemma \ref{lem3.3} below we approximate $Y^\ve$ in the
 following way
 \begin{eqnarray*}
 &|Y^\ve(t)-Y^\ve(s)-\sum_{k=[s/\Del(\ve)]-1}^{[t/\Del(\ve)]-1}\big(\ve\eta^\ve_k(X^\ve((k-1)\Del(\ve)))
 +\ve^2\zeta_k^\ve(X^\ve((k-1)\Del(\ve)))\big)|\\
 &\to 0\,\,\mbox{as}\,\,\ve\to 0,
 \end{eqnarray*}
 where $[\cdot]$ denotes the integral part of a number,
 \begin{eqnarray*}
 &\eta^\ve_k(x)=\int_{k\ve^{-(1-\ka)}}^{(k+1)\ve^{-(1-\ka)}}F(x,\xi(u))du,\\
 & \zeta^\ve_k(x)=\int_{k\ve^{-(1-\ka)}}^{(k+1)\ve^{-(1-\ka)}}du\int_{(k-1)\ve^{-(1-\ka)}}^u
 \langle\nabla_xF(x,\xi(u)),\, B(x,\xi(v))-B(x)\rangle dv,
 \end{eqnarray*}
 $F(x,\xi(u))=\langle H(x),\, B(x,\xi(u))-B(x)\rangle$, $\Del(\ve)=\ve^{1+\ka}$ with $1>\ka>1/2$
 and $\langle\cdot,\cdot\rangle$ denotes the inner product. Now we proceed similarly to the proof
 in the present paper so that asymptotically $\ve\eta^\ve_k$ and $\ve^2\eta_k^\ve$ give rise to
 the diffusion and the drift terms, respectively, whose precise form can be found in \cite{BF} (where
 only the weak convergence was established). Of course, under (\ref{1.2}) the averaged system
is trivial and any function $H$ is an integral of motion since the system does not move. If $H$
is smooth then Theorem \ref{thm2.1} gives an estimate for the uniform approximation error
$E\sup_{0\leq t\leq T}| H(X^\ve(t))-H(\Xi(t))|^{2M}$ where $H(\Xi(t))$ can be represented as a
diffusion across the level curves of $H$.
\end{remark}

\section{Preliminary estimates  }\label{sec3}\setcounter{equation}{0}

In order to benefit from our weak dependence assumptions (\ref{2.2}) we will employ
throughout this paper the following well known result (see, for instance, Corollary
to Lemma 2.1 in \cite{Kha} or Lemma 1.3.10 in \cite{HK}).
\begin{lemma}\label{lem3.1}
Let $H(x,\om)$ be a bounded measurable function on the space $(\bbR^d\times\Om,\,\cB\times\cF)$,
where $\cB$ is the Borel $\sig$-algebra, such that for each $x\in\bbR^d$ the function $H(x,\cdot)$
is measurable with respect to a $\sig$-algebra $\cG\subset\cF$. Let $V$ be an $\bbR^d$-valued
random vector measurable with respect to another $\sig$-algebra $\cH\subset\cF$.
Then with probability one,
\begin{equation}\label{3.1}
 |E(H(V,\om)|\cH)-h(V)|\leq 2\phi(\cG,\cH)\| H\|_{\infty}
 \end{equation}
where $h(x)=EH(x,\cdot)$ and the $\phi$-dependence coefficient was defined in (\ref{2.0}).
In particular (which is essentially an equivalent statement), let
 $H(x_1,x_2),\, x_i\in\bbR^{d_i},\, i=1,2$ be a bounded Borel function and $V_i$ be
 $\bbR^{d_i}$-valued $\cG_i$-measurable random vectors, $i=1,2$ where $\cG_1,\cG_2\subset\cF$ are
 sub $\sig$-algebras. Then with probability one,
 \begin{equation*}
 |E(H(V_1,V_2)|\cG_1)-h(V_1)|\leq 2\phi(\cG_1,\cG_2)\| H\|_{\infty}.
 \end{equation*}
 \end{lemma}

The following lemma shows that the definitions (\ref{2.6}) and (\ref{2.7}) of
 the functions $c(x)$ and $a_{jk}(x)$ are legitimate and it estimates also the
 speed of convergence which will be needed for comparison of characteristic functions
 later on.
 \begin{lemma}\label{lem3.2}
 The limits (\ref{2.6}) and (\ref{2.7}) exist uniformly in $s$ and for all $s,t\geq 0$,
 \begin{equation}\label{3.2}
 |tc(x)-\int_s^{s+t}du\int_{s-t}^uc(x,u,v)dv|\leq 2L^2\int_0^tdu\int_{t+u}^\infty\phi(r)dr
\end{equation}
and
 \begin{equation}\label{3.3}
 |ta_{jk}(x)-\int_s^{s+t}\int_s^{s+t}a_{jk}(x,u,v)dudv|
 \leq 2L^2\int_0^tdu\int_{t+u}^\infty\phi(r)dr.
\end{equation}
  Moreover, $c(x)$ and $b(x)=Eb(x,\xi(0))$ are once and $a_{jk}(x)$ is twice differentiable
  for $j,k=1,...,d$ and for all $x\in\bbR^d$,
 \begin{eqnarray}\label{3.4}
 &|b(x)|\leq L,\,|\nabla_xb(x)|\leq L,\,\max(|c(x)|,\,|a_{jk}(x)|)\leq
 \hat L=2L^2\int_0^\infty\phi(r)dr,\\
 &\max\big(|\nabla_xc(x)|,\,|\nabla_xa_{jk}(x)|,\,|\nabla^2_xa_{jk}(x)|\big)
 \leq 8\hat L\nonumber
\end{eqnarray}
where $L$ is the same as in (\ref{2.3}).
\end{lemma}
\begin{proof}
Observe, first, that (\ref{1.2}) implies also that
\begin{equation}\label{3.5}
E(\nabla_xB(x,\xi(s)))=\nabla_xE(B(x,\xi(s))=0.
\end{equation}
Indeed, let $\bar\Del_i$ be the vector with all zero components except that the component number $i$
is $\Del_i$. Then by (\ref{2.3}),
\[
|\Del_i|^{-1}|B(x+\bar\Del_i,y)-B(x,y)|\leq L,
\]
and so by (\ref{1.2}) and the Lebesgue dominated convergence theorem
\[
0=\lim_{\Del_i\to 0}E((\Del_i)^{-1}(B(x+\bar\Del_i,y)-B(x,y))=E(\frac{\partial B(x,y)}{\partial x_i}).
\]
This together with (\ref{1.2}), (\ref{2.1}), (\ref{2.3}) and Lemma \ref{lem3.1} yields for $v\geq u$,
\begin{equation}\label{3.6}
|c(x,u,v)|=|E\big( \nabla_xB(x,\xi(u))E(B(x,\xi(v))|\cF_{-\infty,u})\big)|\leq 2L^2\phi(|u-v|)
\end{equation}
and the same estimate holds true when $u\geq v$. Similarly,
\begin{equation}\label{3.7}
|a_{jk}(x,u,v)|\leq 2L^2\phi(|u-v|).
\end{equation}
By (\ref{3.6}) and the stationarity of the process $\xi$,
\begin{equation}\label{3.8}
\int_{-t}^u|c(x,u,v)|dv=\int_{-(t+u)}^0|c(x,0,r)|dr\leq 2L^2\int_0^\infty\phi(r)dr<\infty.
\end{equation}
Hence, both the limit
\[
\lim_{t\to\infty}\int_{-t}^u|c(x,u,v)|dv=\int_{-\infty}^0|c(x,0,r)|dr
\]
and the limit
\[
\lim_{t\to\infty}\int_{-t}^uc(x,u,v)dv=\int_{-\infty}^0c(x,0,r)dr
\]
exist. It follows that
\begin{equation}\label{3.9}
c(x)=\int_{-\infty}^0c(x,0,r)dr=\int_{-\infty}^uc(x,u,v)dv
\end{equation}
and for any $u$,
\begin{equation}\label{3.10}
|c(x)-\int_{-t}^uc(x,u,v)dv|=|c(x)-\int_{-(t+u)}^0c(x,0,r)dr|
\leq 2L^2\int_{t+u}^\infty\phi(r)dr
\end{equation}
implying (\ref{2.6}) and (\ref{3.2}). Since,
\[
\int_s^{s+t}du\int_{s-t}^uc(x,u,v)dv=\int_0^tdu\int_{-t}^uc(x,u,v)dv
\]
the limit in (\ref{2.6}) does not depend on $s$ and obviously uniform in $s$.

Next, by the stationarity of the process $\xi$,
\begin{eqnarray*}
&\int_0^t\int_0^ta_{jk}(x,u,v)dudv=\int_0^tdu\int_0^ua_{jk}(x,u,v)dv+\int_0^tdv\int_0^va_{jk}(x,u,v)du\\
&=\int_0^tdu\int_0^ua_{jk}(x,r,0)dr+\int_0^tdv\int_0^va_{jk}(x,0,r)dr.
\end{eqnarray*}
Hence, in the same way as above we conclude that the limit (\ref{2.7}) exists uniformly in $s$,
the estimate (\ref{3.3}) holds true and
\begin{equation}\label{3.11}
a_{jk}(x)=\int_0^\infty a_{jk}(x,r,0)dr+\int_0^\infty a_{jk}(x,0,r)dr=\int_0^\infty(a_{jk}(x,r,0)
+a_{kj}(x,r,0))dr,
\end{equation}
since $a_{jk}(x,u,v)=a_{kj}(x,v,u)$.

Next, the bounds for $b,\, c$ and $a_{jk}$ themselves follow directly from (\ref{2.3}) while
the bounds for their derivatives follow from (\ref{2.3}) and the dominated convergence theorem
in the following way. Consider again the vector $\bar\Del_i$ having all zero components except
for the $i$-th component equal $\Del_i$. By (\ref{2.3}),
\[
|\Del_i|^{-1}|b(x+\bar\Del_i,\xi(s))-b(x,\xi(s))|\leq L\,\mbox{and}\,
|\Del_i|^{-1}|\nabla_xB(x+\bar\Del_i,\xi(s))-\nabla_xB(x,\xi(s))|\leq L
\]
which together with the dominated convergence theorem yields that the limit as $\Del_i\to 0$
and the expectation are interchangeable, and so
\begin{equation}\label{3.12}
|\nabla_xb(x)|\leq E|\nabla_xb(x,\xi(s))|\leq L
\end{equation}
and in addition to (\ref{3.5}) we have also
\begin{equation}\label{3.13}
E\nabla_x^2B(x,\xi(s))=\nabla_x^2EB(x,\xi(s))=0.
\end{equation}
It follows from (\ref{2.1}), (\ref{2.3}), (\ref{3.5}), (\ref{3.12}), (\ref{3.13}) and Lemma \ref{lem3.1}
 similarly to (\ref{3.6}) that
\begin{eqnarray}\label{3.14}
&|\nabla_xc(x,u,v)|\leq 4L^2\phi(|u-v|),\,\,|\nabla_xa_{jk}(x,u,v)|\leq 4L^2\phi(|u-v|)\\
&\mbox{and}\,\,|\nabla^2_xa_{jk}(x,u,v)|\leq 8L^2\phi(|u-v|).\nonumber
\end{eqnarray}
This together with (\ref{2.2}), (\ref{3.9}), (\ref{3.11}) and the dominated convergence
theorem yields that $c(x)$ is once and $a_{jk}x)$ is twice differentiable with the derivatives bounds given by (\ref{3.4}).
\end{proof}

The following step appears already in \cite{Kha}. We freeze the first argument in $B$ and $b$ at certain times and estimate the corresponding error. The new process has terms which will enable us to deal with them as with integrals of a process with weakly dependent terms. Set $\Del=\Del(\ve)=\ve^{1+\ka}$ where $1>\ka>1/2$ and we introduce also
 $X^\ve_k=X^\ve(\Del(\ve)k),\, k=0,1,...,[T/\Del(\ve)],\,
  t_k=t_k(\ve)=k\ve^{-(1-\ka)}=k\Del(\ve)\ve^{-2}$ and $\al_k^\ve=\al_k^\ve(X^\ve_{k-1}),\,
  \be_k^\ve=\be_k^\ve(X^\ve_{k-1}),\, \gam_k^\ve=\gam_k^\ve(X^\ve_{k-1})$ where
   $X^\ve_{0}=X^\ve_{-1}=x$,
  \begin{eqnarray*}
  &\al^\ve_k(x)=\int_{t_k}^{t_{k+1}}B(x,\xi(u))du,\, \be^\ve_k(x)=\int_{t_k}^{t_{k+1}}b(x,\xi(u))du\\ &\mbox{and}\,\,\,\gam_k^\ve(x)=\int_{t_k}^{t_{k+1}}du\int_{t_{k-1}}^u\nabla_xB(x,\xi(u))B(x,\xi(v))dv.
  \end{eqnarray*}
  Introduce the process
  \[
  \breve X^\ve(t)=x_0+\sum_{k=0}^{[t/\Del(\ve)]-1}(\ve\al_k^\ve+\ve^2\be^\ve_k+\ve^2\gam_k^\ve),\,\,
  x_0=X^\ve(0).
  \]
 \begin{lemma}\label{lem3.3} For any $T\geq t>s\geq 0$,
 \begin{eqnarray}\label{3.15}
 &\big\vert X^\ve(t)-X^\ve(s)-\breve X^\ve(t)+\breve X^\ve(s)\big\vert\\
& \leq L^2T\ve^{2\ka-1}(1+\ve)\big(\frac 76L(1+\ve)+\frac 32\ve^{1-\ka}(1+L(1+2\ve))\big)+2L\ve^{\ka}(1+\ve).\nonumber
 \end{eqnarray}
 \end{lemma}
 \begin{proof}
 Changing variables, for any $0\leq u\leq v\leq T/\ve^2$ we write
 \[
X^\ve(\ve^2v)=X^\ve(\ve^2u)+\ve\int_u^v\big( B(X^\ve(\ve^2w),\xi(w))+\ve b(X^\ve(\ve^2),\xi(w))\big)dw.
\]
 Now by the Taylor formula (with two terms for $B$ and one term for $b$)
 and the reminder
 \begin{eqnarray}\label{3.16}
 &X^\ve(\Del(\ve)(k+1))-X^\ve(\Del(\ve)k)
 =\ve\int_{t_k}^{t_{k+1}}\big(B(X^\ve(\ve^2v),\xi(v))\\
 &+ \ve b(X^\ve(\ve^2v),\xi(v))\big)dv
 =\ve\int_{t_k}^{t_{k+1}}\big(B(X^\ve_{k-1},\xi(v))+\ve b(X^\ve_{k-1},\xi(v))\big)dv
 \nonumber\\
 &+\ve\int_{t_k}^{t_{k+1}}\nabla_xB(X^\ve_{k-1},\xi(v))(X^\ve(\ve^2v)-X^\ve_{k-1})dv+\ve R^\ve_{1,k}
 \nonumber\\
 &=\ve\al_k^\ve+\ve^2(\be_k^\ve+\gam^\ve_k)+\ve R^\ve_k\nonumber
 \end{eqnarray}
 where we use that
 \begin{eqnarray*}
 & R^\ve_k=R^\ve_{1,k}+\ve R^\ve_{2,k},\,\,\int_{t_k}^{t_{k+1}}\nabla_xB(X^\ve_{k-1},\xi(u))
 (X^\ve(\ve^2u)-X^\ve_{k-1})du\\
 &=\ve\int_{t_k}^{t_{k+1}}du\nabla_xB(X^\ve_{k-1},\xi(u))\int_{t_{k-1}}^uB(X^\ve_{k-1},\xi(v))dv
 +\ve R^\ve_{2,k}
 \end{eqnarray*}
 and by (\ref{2.3}),
 \begin{eqnarray*}
 &|R^\ve_{1,k}|
 \leq\int_{t_k}^{t_{k+1}}(\ve\sup_x|\nabla_xb(x,\xi(v))||X^\ve(\ve^2u)
 -X^\ve_{k-1}|\\
 &+\frac 12\sup_x|\nabla_x^2B(x,\xi(v))||X^\ve(\ve^2u)-X^\ve_{k-1}|^2)du\\
 &\leq L\int_{t_k}^{t_{k+1}}(\ve|X^\ve(\ve^2u)-X^\ve_{k-1}|+\frac 12|X^\ve(\ve^2u)-X^\ve_{k-1}|^2)du\\
 &\leq\ve^2L^2(1+\ve)\int_{t_k}^{t_{k+1}}((u-t_{k-1})+\frac 12(1+\ve)L(u-t_{k-1})^2)du\\
 &\leq L^2(1+\ve)(\frac 32\ve^{2\ka}+ \frac 76L(1+\ve)\ve^{3\ka-1})
 \end{eqnarray*}
 while
 \begin{eqnarray*}
 &|R^\ve_{2,k}|=\big\vert\int_{t_k}^{t_{k+1}}\nabla_xB(X^\ve_{k-1},\xi(u))du\int_{t_{k-1}}^u
 \big(B(X^\ve(\ve^2v),\xi(v))-B(X^\ve_{k-1},\xi(v))\\
 &+\ve b(X^\ve(\ve^2v),\xi(v))\big)dv\big\vert
\leq L^2\int_{t_k}^{t_{k+1}}\int_{t_{k-1}}^u(|X^\ve(\ve^2v)-X^\ve_{k-1}|+\ve)dudv\\
&\leq L^3\ve(1+2\ve)\int_{t_k}^{t_{k+1}}(u-t_{k-1})du=\frac 32L^3\ve^{2\ka-1}(1+2\ve).
 \end{eqnarray*}
 Now summing in $k$ from $[s/\Del(\ve)]$ to $[t/\Del(\ve)]-1$ and taking into account that
 for any $u\geq 0$,
 \[
 |X^\ve(u)-X^\ve([u/\Del(\ve)]\Del(\ve))|\leq\frac 1\ve(1+\ve)L(u-[u/\Del(\ve)]\Del(\ve))
 \leq L\ve^\ka(1+\ve),
 \]
  we obtain (\ref{3.15}).
  \end{proof}

  We will employ several times the following general moment estimate which appears as Lemma 3.2.5 in \cite{HK} for random variables and we refer the readers there for its proof providing here only its extension to random vectors. This and the following lemma will
  be needed, in particular, for characteristic functions estimates below which is crucial
  for the strong approximation theorem in the next section.
  \begin{lemma}\label{lem3.4}
  Let $(\Om,\cF,P)$ be a probability space with a filtration of $\sig$-algebras $\cG_j,\, j\geq 1$ and
  a sequence of random $d$-dimensional vectors $\eta_j,\, j\geq 1$ such that $\eta_j$
  is $\cG_j$-measurable, $j=1,2,...$. Suppose that for some integer $M\geq 1$,
  \[
  A_{2M}=\sup_{i\geq 1}\sum_{j\geq i}\| E(\eta_j|\cG_i)\|_{2M}<\infty
  \]
  where $\|\eta\|_p=(E|\eta|^p)^{1/p}$ and $|\eta|$ is the Euclidean norm of a (random) vector $\eta$.
  Then for any integer $n\geq 1$,
  \[
  E|\sum_{j=1}^n\eta_j|^{2M}\leq 3(2M)!d^MA_{2M}^{2M}n^M.
  \]
  \end{lemma}
  \begin{proof}
  Let $\eta_j=(\eta_{j1},...,\eta_{jd})$. Then
  \[
  A_{2M}^{(l)}=\sup_{i\geq 1}\sum_{j\geq i}\| E(\eta_{jl}|\cG_i)\|_{2M}\leq A_{2M}
  \]
  since $|E(\eta_j|\cG_i)|\geq |E(\eta_{jl}|\cG_i)|$ for each $l=1,...,d$. Hence, by the
  $d=1$ version of the above lemma appeared as Lemma 3.2.5 in \cite{HK},
  \[
  E(\sum_{j=1}^n\eta_{jl})^{2M}\leq 3(2M)!A_{2M}^{2M}n^M,
  \]
  and so
  \[
  E|\sum_{j=1}^n\eta_j|^{2M}=E|\sum_{l=1}^d(\sum_{j=1}^n\eta_{jl})^2|^M\leq d^{M-1}\sum_{l=1}^d
  (\sum_{j=1}^n\eta_{jl})^{2M}\leq 3(2M)!d^MA_{2M}^{2M}n^M
  \]
  completing the proof.
  \end{proof}

  We will use the following moment estimate
  \begin{lemma}\label{lem3.5}
  For any $t\geq 0$, $x\in\bbR^d$ and an integer $M\geq 1$,
  \begin{equation}\label{3.17}
  E|\int_0^tB(x,\xi(u))du|^{2M}\leq C_1(M)t^M
  \end{equation}
  where $C_1(M)>0$ appears at the end of the proof.
  \end{lemma}
  \begin{proof}
  First, we write
  \begin{eqnarray}\label{3.18}
  &|\int_0^tB(x,\xi(u))du|^{2M}\leq\big(\sum_{i=1}^d|\int_0^tB_i(x,\xi(u))du|\big)^{2M}\\
  &\leq d^{2M-1}\sum_{i=1}^d|\int_0^tB_i(x,\xi(u))du|^{2M}.\nonumber
  \end{eqnarray}
  Set $\zeta_m=\zeta_m(x)=\int_{m-1}^mB(x,\xi(u))du,\, m=1,2,...$ which is a stationary
   in $m$ sequence of random vectors. Since by (\ref{2.3}),
   \[
   |\int_0^tB(x,\xi(u))du-\sum_{m=1}^{[t]}\zeta_m|\leq L,
   \]
   we can write
   \begin{equation}\label{3.19}
   |\int_0^tB(x,\xi(u))du|^{2M}\leq 2^{2^{M-1}}(|\sum_{m=1}^{[t]}\zeta_m|^{2M}+L^{2M}).
   \end{equation}
    Set $S_n=\sum_{m=1}^n\zeta_m$, denote $\cG_{m}=\cF_{-\infty,m}$ and observe that $\zeta_{m}$ is $\cG_m$-measurable. By (\ref{1.2}) and the definition (\ref{2.0})--(\ref{2.1}) of
    the coefficient $\phi$ for $m>k$,
    \begin{equation}\label{3.24}
    |E(\zeta_{m}|\cG_k)|\leq\int_{m-1}^m|E(B(x,\xi(u))|\cG_k)|du\leq 2L\phi(m-k-1)
     \end{equation}
     while for $m=k$ we estimate the left hand side of (\ref{3.24}) just by $L$.
  It follows from (\ref{2.3}) and (\ref{3.24}) that
   \begin{equation}\label{3.25}
   A_{2M}=\sup_{k\geq 1}\sum_{m\geq k}\| E(\zeta_{m1}|\cG_k)\|_{2M}\leq 2L(1+\sum_{l=0}^\infty\phi(l))
   \end{equation}
   where $\|\cdot\|_p$ denotes the $L^p$-norm.

 Applying Lemma \ref{lem3.4} we obtain from (\ref{3.25}) that for all $r\geq 0$,
 \begin{equation}\label{3.27}
 E|S_{n}|^{2M}\leq 3(2M)!d^MA_{2M}^{2M}n^M.
 \end{equation}
 Setting $n=[t]$ here we obtain (\ref{3.17}) taking into account (\ref{3.18}) and (\ref{3.19}) with
 \[
  C_1(M)=(2d)^{2M+1}(2M)!\big(2L(1+\sum_{l=0}^\infty\phi(l))\big)^{2M}.
  \]
  \end{proof}

 Next, for each $t>0$ and $x\in\bbR^d$ introduce the characteristic function
 \[
 f_t(x,w)=E\exp(i\langle w,t^{-1/2}\int_0^tB(x,\xi(u))du\rangle),\,\, w\in\bbR^d
 \]
 where $\langle\cdot,\cdot\rangle$ denotes the inner product. We will need the following estimate.

 \begin{lemma}\label{lem3.6} For any $t>0$ and $x\in\cR^d$,
 \begin{equation}\label{3.28}
 |f_t(x,w)-\exp(-\frac 12\langle A(x)w,w\rangle)|\leq C_2t^{-\wp}
 \end{equation}
 for all $w\in\bbR^d$ with $|w|\leq t^{\wp/2}$ where we can take any $\wp\leq\frac 1{20}$ and
 $C_2>0$ appearing at the end of the proof.
 \end{lemma}
 \begin{proof}
 The left hand side of (\ref{3.28}) does not exceed $2$ and for $t<16$ we estimate it by
  $2(16)^\wp t^{-\wp}$ which is not less. So, in what follows, we will assume that $t\geq 16$.
  In order to obtain explicit constants and for completeness we will provide a detailed proof here which employs the standard block-gap technique rather than relying on one of known results such
  as Theorem 3.23 in \cite{DP}.
 Set $n(t)=[t(t^{3/4}+t^{1/4})^{-1}],\, q_k(t)=k(t^{3/4}+t^{1/4}),\, r_k(t)=q_{k-1}(t)+t^{3/4}$
  for $k=1,2,...,n(t)$ with $q_0(t)=0$. Next, we introduce for $k=1,2,...,n(t)$,
  \[
  y_k=y_k(t)=\int_{q_{k-1}(t)}^{r_k(t)}B(x,\xi(u))du,\,\,\,
  z_k=z_k(t)=\int^{q_{k}(t)}_{r_k(t)}B(x,\xi(u))du
  \]
  and $z_{n(t)+1}=\int_{q_{n(t)}(t)}^tB(x,\xi(u))du$. Then by Lemma \ref{lem3.5},
  \begin{eqnarray}\label{3.29}
  &E|\sum_{1\leq k\leq n(t)+1}z_k|^2\leq 2E|\sum_{1\leq k\leq n(t)}z_k|^2+2E|z_{n(t)+1}|^2\\
  &\leq 2n(t)\sum_{1\leq k\leq n(t)}E|z_k|^2+2E|z_{n(t)+1}|^2\nonumber\\
  &\leq 2C_1(1)((n(t))^2t^{1/4}+t^{3/4})\leq 4C_1t^{3/4}.\nonumber
  \end{eqnarray}

  Next, by (\ref{3.29}) and the Cauchy-Schwarz inequality,
  \begin{eqnarray}\label{3.30}
  &|f_t(x,w)-E\exp(i\langle w,t^{-1/2}\sum_{1\leq k\leq n(t)}y_k\rangle)|\\
  &\leq E|\exp(i\langle w,t^{-1/2}\sum_{1\leq k\leq n(t)+1}z_k\rangle)-1|\leq t^{-1/2}E\langle w,\sum_{1\leq k\leq n(t)+1}z_k\rangle\nonumber\\
  &\leq t^{-1/2}|w|E|\sum_{1\leq k\leq n(t)+1}z_k|\leq 2\sqrt {C_1(1)}|w|t^{-1/8}\nonumber
  \end{eqnarray}
  where we use that for any real $a,b$,
  \[
  |e^{i(a+b)}-e^{ib}|=|e^{ia}-1|\leq |a|.
  \]
  We will obtain (\ref{3.28}) from (\ref{3.30}) by estimating
  \begin{equation}\label{3.31}
  |E\exp(i\sum_{1\leq k\leq n(t)}\eta_k)-\exp(-\frac 12\langle A(x)w,w\rangle)|\leq I_1+I_2
  \end{equation}
  where
  \begin{eqnarray*}
  &\eta_k=\langle w,t^{-1/2}y_k\rangle,\,\,\,
  I_1=|E\exp(i\sum_{1\leq k\leq n(t)}\eta_k)-\prod_{1\leq k\leq n(t)}Ee^{i\eta_k}|\\
  &\mbox{and}\,\,\, I_2=|\prod_{1\leq k\leq n(t)}Ee^{i\eta_k}-\exp(-\frac 12\langle A(x)w,w\rangle)|.
  \end{eqnarray*}

  By Lemma \ref{lem3.1},
  \begin{eqnarray}\label{3.32}
  &I_1\leq\sum_{m=2}^{n(t)}\big(|\prod_{m+1\leq k\leq n(t)}Ee^{i\eta_k}|\\
  &\times|E\exp(i\sum_{1\leq k\leq m}\eta_k)- E\exp(i\sum_{1\leq k\leq m-1}
  \eta_k)Ee^{i\eta_m}|\big)\nonumber\\
  &\leq\sum_{m=2}^{n(t)}|E\exp(i\sum_{1\leq k\leq m}\eta_k)-E\exp(i\sum_{1\leq k\leq m-1}\eta_k)
  Ee^{i\eta_m}|\nonumber\\
  &\leq n(t)\phi(t^{1/4})\leq\sup_{r>0}(r\phi(r)).\nonumber
  \end{eqnarray}
  where $\prod_{n(t)+1\leq k\leq n(t)}=1$.

  In order to estimate $I_2$ we observe that
  \[
  |\prod_{1\leq j\leq l}a_j-\prod_{1\leq j\leq l}b_j|\leq\sum_{1\leq j\leq l}|a_j-b_j|
  \]
  whenever $0\leq |a_j|, |b_j|\leq 1,\, j=1,...,l$, and so
  \begin{eqnarray}\label{3.37}
  &I_2\leq\sum_{1\leq k\leq n(t)}|Ee^{i\eta_k}-\exp(-\frac 1{2n(t)}\langle A(x)w,w\rangle)|\\
  &\leq\frac 12\sum_{1\leq k\leq n(t)}|E\eta_k^2-\frac 1{n(t)}\langle A(x)w,w\rangle|\nonumber\\
  &+ \sum_{1\leq k\leq n(t)}E|\eta_k|^3+ \frac 1{4n(t)}|\langle A(x)w,w\rangle|^2\nonumber
 \end{eqnarray}
 where we use (\ref{1.2}) and that for any real $a$,
 \[
 |e^{ia}-1-ia+\frac {a^2}2|\leq |a|^3\,\,\mbox{and}\,\, |e^{-a}-1+a|\leq a^2\,\,\mbox{if}\,\, a\geq 0.
 \]

 Now,
 \begin{eqnarray*}
 &E\eta_k^2=t^{-1}E(\sum_{j=1}^dw_j\int_{q_{k-1}(t)}^{r_k(t)}B_j(x,\xi(u))du)^2\\
 &=t^{-1}\sum_{j,l=1}^d
 w_jw_l\int_{q_{k-1}(t)}^{r_k(t)}\int_{q_{k-1}(t)}^{r_k(t)}a_{jl}(x,u,v)dudv.
 \end{eqnarray*}
 Hence, by (\ref{3.3}),
 \begin{equation}\label{3.38}
 |E\eta_k^2-t^{-1/4}\langle A(x)w,w\rangle|\leq 2L^2d|w|^2t^{-1}\int_0^{t^{3/4}}du\int^\infty_{t^{3/4}+u} \phi(r)dr.
 \end{equation}
 By (\ref{3.7}) and (\ref{3.11}) we have also
 \begin{eqnarray}\label{3.39}
 &|(\frac 1{n(t)}-t^{-1/4})\langle A(x)w,w\rangle|\\
 &\leq 12L^2d(\int_0^\infty\phi(r)dr)|w|^2([\frac t{t^{3/4}+t^{1/4}}]^{-1}-t^{-1/4}).\nonumber
 \end{eqnarray}
 Since we assume that $t\geq 16$,
 \begin{eqnarray}\label{3.40}
 &[\frac t{t^{3/4}+t^{1/4}}]^{-1}-t^{-1/4}\leq (\frac t{t^{3/4}+t^{1/4}}-1)^{-1}-t^{-1/4}\\
 &=t^{-1/2}\frac {1+t^{-1/4}+t^{-1/2}}{1-t^{-1/4}-t^{-3/4}}\leq 8t^{-1/2}.\nonumber
 \end{eqnarray}
 By Lemma \ref{lem3.5}, H\" older inequality and the stationarity of the process $\xi$,
 \begin{equation}\label{3.41}
 E|\eta_k|^3\leq t^{-3/2}|w|^3\big(E(\int_{q_{k-1}(t)}^{r_k(t)}B(x,\xi(u))du)^4\big)^{3/4}\leq C_1^{3/4}(2)t^{-3/8}|w|^3.
 \end{equation}
 Again, by (\ref{3.7}) and (\ref{3.11}),
 \begin{equation}\label{3.42}
 \frac 1{n(t)}|\langle A(x)w,w\rangle|^2\leq 80L^2d^2t^{-1/4}|w|^4(\int_0^\infty\phi(r)dr)^2.
 \end{equation}
 Now, collecting (\ref{3.37})--(\ref{3.42}) we obtain that
 \begin{eqnarray}\label{3.43}
 &I_2\leq L^2d|w|^2t^{-3/4}\int_0^{t^{3/4}}du\int_{t^{3/4}+u}^\infty\phi(r)dr\\
 &+16L^2d|w|^2t^{-1/4}\int_0^\infty\phi(r)dr(3+5d|w|^2\int_0^\infty\phi(r)dr)
 +C_1^{3/4}(2)t^{-1/8}|w|^3.\nonumber
 \end{eqnarray}
 Finally, (\ref{3.30}), (\ref{3.31}), (\ref{3.32}) and (\ref{3.43}) yield (\ref{3.28})
 with
  \begin{eqnarray*}
 &C_2=2(16)^\wp+2C_1^{1/2}(1)+C_1^{3/4}(2)+L^2\sup_{r>0}(r\phi(r))+L^2d\sup_{r>0}(r^2(\phi(r))\\
 &+16L^2d\int_0^\infty\phi(r)dr(3+5d\int_0^\infty\phi(r)dr).
  \end{eqnarray*}
  completing the proof.
 \end{proof}

 In order to obtain uniform moment estimates required by Theorem \ref{thm2.1} we will need in
what follows the following general estimate which is based on the martingale approximation technique.
\begin{lemma}\label{lem4.1} Let $\eta_1,\eta_2,...,\eta_N$ be random $d$-dimensional vectors and
$\cH_1\subset\cH_2\subset...\subset\cH_N$ be a filtration of $\sig$-algebras such that $\eta_m$ is
$\cH_m$-measurable for each $m=1,2,...,N$. Assume also that $E|\eta_m|^q<\infty$ for some $q>1$
and each $m=1,...,N$. Set $\Sig_m=\sum_{j=1}^m\eta_j$. Then
\begin{equation}\label{4.1}
E\max_{1\leq m\leq N}|\Sig_m|^q\leq 2^{q-1}\big((\frac q{q-1})^qE|\Sig_N|^q+E\max_{1\leq m
\leq N-1}|\sum^N_{j=m+1}E(\eta_j|\cH_m)|^q\big).
\end{equation}
\end{lemma}
\begin{proof}
Set $M_m=\Sig_m+\sum_{j=m+1}^NE(\eta_j|\cH_m)$ for $m=1,...,N-1$ and $M_N=S_N$. Then, $M_m$ is
$\cH_m$-measurable and $E(M_m|\cH_{m-1})=M_{m-1}$, i.e. $\{ M_m\}_{m=1}^N$ is a martingale with
respect to the filtration $\{\cH_m\}_{m=1}^N$. It follows from the definition above,
\[
E\max_{1\leq m\leq N}|\Sig_m|^q\leq 2^{q-1}\big(E\max_{1\leq m\leq N}|M_m|^q+E\max_{1\leq m
\leq N-1}|\sum^N_{j=m+1}E(\eta_j|\cH_m)|^q\big).
\]
By the Doob submartingale inequality
\[
E\max_{1\leq m\leq N}|M_m|^q\leq (\frac q{q-1})^qE|M_N|^q=(\frac q{q-1})^qE|\Sig_N|^q
\]
and (\ref{4.1}) follows.
\end{proof}

\section{Strong approximations  }\label{sec4}\setcounter{equation}{0}
\subsection{Another block-gap partition}\label{subsec4.1}
Next, each time interval $[t_{k-1}(\ve),\, t_k(\ve)]$ will be split into a block and a gap
before it in the following way (where, recall, $t_k=t_k(\ve)$ was defined before Lemma \ref{lem3.3}). Set $s_k=t_{k-1}+\ve^{-\frac 14(1-\ka)}$,
\begin{eqnarray*}
& Y^\ve_{k}(x)=\int_{t_{k-1}}^{t_{k}}B(x,\xi(u))du,\,\, Z^\ve_{k}(x)=
\int_{s_{k}}^{t_{k}}B(x,\xi(u))du,\\
&\mbox{and}\,\, Y^\ve(x,t)=\sum_{k=1}^{k(\ve,t)}Y^\ve_k(X^\ve_{k-2})\\
\end{eqnarray*}
where $X^\ve_0=X^\ve_{-1}=x$ and $k(\ve,t)=\max\{ k:\, t_k\leq t\}=[\ve^{(1-\ka)}t]$. Then
\begin{equation}\label{4.1}
Y^\ve(x,t)=\sum_{0\leq k<k(\ve,t)}\al^\ve_k
\end{equation}
where $\al^\ve_k$ was defined before Lemma \ref{lem3.3}.

For the strong approximation theorem below we will need the following corollary of Lemma \ref{lem3.6},
\begin{lemma}\label{lem4.3} For any $\ve>0$ and $k\leq T\ve^{-(1+\ka)}$,
with probability one,
\begin{eqnarray}\label{4.15}
&\big\vert E\big(\exp(i\langle w,\, (t_{k}-t_{k-1})^{-1/2}Y^\ve_{k}(X^\ve_{k-2})\rangle)|
\cF_{-\infty,t_{k-1}})-g_{X^\ve_{k-2}}(w)\big\vert\\
&\leq C_2(t_{k}-s_{k})^{-\wp}+2\phi(s_k-t_{k-1})+L(s_k-t_{k-1})(t_k-t_{k-1})^{-(1-\wp)/2}\nonumber\\
&+\hat Ld(s_k-t_{k-1})(t_k-t_{k-1})^{-(1-\wp)}\leq C_2\ve^{\wp(1-\ka)}(1-\ve^{3(1-\ka)/4})^{-\wp}\nonumber\\
&+2\phi(\ve^{-(1-\ka)/4})+L\ve^{(1-2\wp)(1-\ka)/4}+\hat Ld\ve^{(3-4\wp)(1-\ka)/4}\nonumber
\end{eqnarray}
for all $w\in\bbR^d$ with $|w|\leq (t_{k}-s_{k})^{\wp/2}$ where $g_x(w)=\exp(-\frac 12\langle A(x)w,w\rangle)$, $C_2$ and $\wp$ are from Lemma \ref{lem3.6} and $\hat L$ was defined in Lemma \ref{lem3.2}.
\end{lemma}
\begin{proof}
First, we write
\begin{eqnarray*}
&\big\vert E\big(\exp(i\langle w,\, (t_{k}-t_{k-1})^{-1/2}Y^\ve_{k}(X^\ve_{k-2})\rangle)|
\cF_{-\infty,t_{k-1}})-g_{X^\ve_{k-2}}(w)\big\vert\\
&\leq R_1(w)+R_2(w)+R_3(w)
\end{eqnarray*}
where
\begin{eqnarray*}
&R_1(w)=\big\vert E\big(\exp(i\langle w,\, (t_{k}-t_{k-1})^{-1/2}Z^\ve_{k}(X^\ve_{k-2})\rangle)|
\cF_{-\infty,t_{k-1}})\\
&-g_{X^\ve_{k-2}}((\frac {t_k-s_k}{t_k-t_{k-1}})^{1/2}w)\big\vert,
\end{eqnarray*}
\[
R_2(w)=E\big(|\exp(i\langle w,\, (t_{k}-t_{k-1})^{-1/2}(Y^\ve_{k}(X^\ve_{k-2})-Z^\ve_{k}(X^\ve_{k-2}))\rangle)-1|\big\vert
\cF_{-\infty,t_{k-1}})
\]
and
\[
R_3(w)=\sup_x|g_{x}((\frac {t_k-s_k}{t_k-t_{k-1}})^{1/2}w)-g_x(w)|.
\]
By Lemma \ref{3.1} and the stationarity of the process $\xi$,
\[
R_1(w)\leq C_2(t_{k}-s_{k})^{-\wp}+2\phi(s_k-t_{k-1})
\]
for all $w\in\bbR^d$ with
\[
|w|\leq (t_{k}-s_{k})^{\wp/2}\leq (t_{k}-s_{k})^{\wp/2}
(\frac {t_k-t_{k-1}}{t_k-s_k})^{1/2}.
\]
For such $w\in\bbR^d$ we have also
\[
R_2(w)\leq L|w|(s_k-t_{k-1})(t_k-t_{k-1})^{-1/2}\leq L(s_k-t_{k-1})(t_k-t_{k-1})^{-(1-\wp)/2}
\]
and
\[
R_3(w)\leq\hat Ld|w|^2(1-(\frac {t_k-s_k}{t_k-t_{k-1}})^{1/2})\leq
\hat Ld(s_k-t_{k-1})(t_k-t_{k-1})^{-(1-\wp)}
\]
yielding (\ref{4.15}).
\end{proof}

\subsection{Strong approximation theorem}\label{subsec4.2}
Our strong approximations will be based on the following result which follows, essentially, from
Theorem 1 in \cite{MP1} but on the final step we have to employ also Lemma A1 from \cite{BP}. As usual, we will denote by $\sig\{\cdot\}$ a $\sig$-algebra generated by random variables or vectors 
 appearing inside the braces and we write $\cG\vee\cH$ for the minimal $\sig$-algebra containing
  both $\sig$-algebras $\cG$ and $\cH$.
\begin{theorem}\label{thm4.4}
Let $V$ and $\Sig$ be an $\bbR^d$-valued random vector and $d\times d$ random symmetric matrix,
respectively, on some probability space $(\Om,\cF,P)$ such that $\Sig$ is measurable with respect to a $\sig$-algebra $\cG$ generated by some $q$-dimensional random vector $Z$, i.e. $\cG=\sig\{ Z\}$.
Assume that the probability space $(\Om,\cF,P)$ is rich enough so that there exists on it a uniformly
distributed on $[0,1]$ random variable independent of the $\sig$-algebra $\cG\vee\sig\{ V\}$. Let
$G$ be a probability distribution on $\bbR^d$ with a characteristic function $f$ and suppose that 
for some non-negative numbers $\nu,\del$ and $K\geq 10^8d$,
 \begin{equation}\label{4.16}
 \int_{|w|\leq K_m}E\big\vert E(\exp(i\langle w,V\rangle)|\cG)-f(\Sig w)\big\vert dw
 \leq\nu(2K)^d
 \end{equation}
 and that
 \begin{equation}\label{4.17}
 E\big(G(\{ x:\, |\Sig x|\geq\frac 12K\}|\cG_{m-1})\big)<\del.
 \end{equation}
 Then $V,Z,U$ and $\Sig$ can be redefined on a richer probability space preserving their joint
 distributions where there exists an $\bbR^d$-valued random vector $W$ independent of $Z$
 (and so also of $\Sig$), measurable with respect to $\sig\{ V,Z,U\}$ and such that $W$ has the distribution $G$ and
 \begin{equation}\label{4.18}
 P\{ |V-\Sig W|\geq\vr\}\leq\vr
 \end{equation}
 where $\vr=16dK^{-1}\log K+2\nu^{1/2}K^d+2\del^{1/2}$.
 In particular, the Prokhorov distance between the distributions $\cL(V)$ and
 $\cL(\Sig W)$ of $V$ and $\Sig W$, respectively, does not exceed $\vr$.
\end{theorem}
\begin{proof}
Let $\hat G(\cdot|\cG)$ be the regular conditional distribution on $\bbR^d$ (see \cite{Dud})
with the conditional characteristic function 
\[
\hat f(w|\cG)=\int_{\bbR^d}exp(i\langle u,x\rangle)\hat G(dx|\cG)=f(\Sig w).
\]
Then the distribution $\hat G$ has the characteristic function $\hat f(w)=Ef(\Sig w)$ which is the
characteristic function of the random vector $\Sig W$ where $W$ is any random vector in $\bbR^d$
independent of
$\cG=\sig\{ Z\}$ and having the distribution $G$. Applying Theorem 1 from \cite{MP1} (see also 
Remark 2.1 there) to our situation we obtain an $\bbR^d$-valued random vector $\hat W$ with the
 distribution $\hat G$ which is measurable with respect to $\sig\{ V,Z,U\}$ and such that
\[
P\{ |V-\hat W|\geq\vr\}\leq\vr
\]
where $\vr$ is the same as in (\ref{4.18}). Next, let $Q$ be the joint distribution of the collection $(V,\Sig,Z,U)$ and of $\hat W$ and let $R$ be the joint distribution of $\Sig W$ and the collection
$(W,\Sig,Z,U)$ where $W$ is an $\bbR^d$-valued random vector independent of $\cG=\sig\{ Z\}$ and
 having the distribution $G$. Observe now that the second marginal of $Q$ and the first marginal of
 $R$ are the same, and so by Lemma A1 from \cite{BP} we can redefine $V,\Sig,Z,U$ and
 $W$ on a richer probability space so that the redefined $V,\Sig,Z,U$ and $W$ (denoted by the same 
 letters) are such that $(V,\Sig,Z,U)$ and $\Sig W$ have the joint distribution $Q$ while
 $\Sig W$ and $(W,\Sig,Z,U)$ have the joint distribution $R$. In particular, (\ref{4.18}) remains
 true, $W$ is independent of $Z$, and so $W$ is independent of $\Sig$ since by our assumption $\Sig$,
 being measurable with respect to $\cG$, can be represented as a Borel function of $Z$ which is 
 preserved after the redefinition. This completes the proof of this theorem.
\end{proof}

Applying Theorem \ref{thm4.4} repeatedly we obtain
\begin{corollary}\label{cor4.4+}
Let $\{ V_m,\, m\geq 1\}$ and $\{\Sig_m,\, m\geq 1\}$
be sequences of random $d$-dimensional vectors and $d\times d$ random symmetric matrices, respectively, defined on some probability space $(\Om,\cF,P)$ and such that $V_m$ and $\Sig_m$ are
measurable with respect to $\cG_m$ where $\cG_m,\, m\geq 1$ is a filtration of countably generated
sub-$\sig$-algebras of $\cF$. Assume that the probability space is rich enough
so that there exists on it a sequence of uniformly distributed on $[0,1]$ independent random variables
 $U_m,\, m\geq 1$ independent of $\vee_{m\geq 1}\cG_m$. Let $G_m,\, m\geq 1$ be a sequence of
 probability distributions on $\bbR^d$ with the characteristic functions $f_m$ and suppose
 that for some non-negative numbers $\nu_m,\del_m$ and $K_m\geq 10^8d$,
 \begin{equation}\label{4.16+}
 \int_{|w|\leq K_m}E\big\vert E(\exp(i\langle w,V_m\rangle)|\cG_{m-1})-f_m(\Sig_{m-1}w)\big\vert dw
 \leq\nu_m(2K_m)^d,
 \end{equation}
 where $\cG_0$ is the trivial $\sig$-algebra and $\Sig_0$ is a constant matrix, and that
 \begin{equation}\label{4.17+}
 E\big(G_m(\{ x:\, |\Sig_{m-1}x|\geq\frac 12K_m\}|\cG_{m-1})\big)<\del_m.
 \end{equation}
 Then there exists a sequence $\{ W_m,\, m\geq 1\}$ of independent $\bbR^d$-valued random vectors defined on $(\Om,\cF,P)$ such that for each $m\geq 1$,

 (i) $W_m$ has the distribution $G_m$;

 (ii) $W_m$ is $\cG_m\vee\sig\{U_m\}$-measurable and $W_m$ is independent of $\cG_{m-1}$ (and so of $U_1,...,U_{m-1}; W_1,...,W_{m-1};\Sig_1,...,\Sig_{m-1}$);

 (iii) Let $\vr_m=16dK^{-1}_m\log K_m+2\nu_m^{1/2}K_m^d+2\del_m^{1/2}$. Then
 \begin{equation}\label{4.18+}
 P\{ |V_m-\Sig_{m-1}W_m|\geq\vr_m\}\leq\vr_m
 \end{equation}
 and, in particular, the Prokhorov distance between the distributions $\cL(V_m)$ and
 $\cL(\Sig_{m-1}W_m)$ of $V_m$ and $\Sig_{m-1}W_m$, respectively, does not exceed $\vr_m$.
 \end{corollary}
 \begin{proof}
We construct $W_m,\, m=1,2,...$ successively. First, $W_1$ is constructed by Theorem \ref{thm4.4}
assuming that $\cG_0$ is the trivial $\sig$-algebra, and so $\Sig_0$ is a constant matrix. Since
$\cG_m,\, m\geq 1$ are countably generated, there exist random variables (vectors) $Z_m,\, m\geq 1$ 
such that $Z_m$ generates $\cG_m$, i.e. $\cG_m=\sig\{ Z_m\}$. Assume
that $W_1,...,W_{m-1}$ are already constructed. Since (\ref{4.16+}) and (\ref{4.17+}) hold true
we can apply Theorem \ref{thm4.4} with $V=V_m$, $Z=Z_{m-1}$, $\cG=\cG_{m-1}$ and $\Sig=\Sig_{m-1}$
to obtain $W=W_m$ which satisfies the conditions of this corollary completing the proof by induction.
\end{proof}

 In order to apply this theorem we set
 $V_{m}=(t_{m}-t_{m-1})^{-1/2}Y^\ve_{m}(X^\ve_{m-2})$, $\Sig_m=\sig(X^\ve_{m-1})$,
 $\cG_{m}= \cF_{-\infty,t_{m}}$ and $f_{m}(w)=\exp(-\frac 12\langle w,w\rangle)$ so that
 $f_m(\sig(x)w)=g_x(w)$ where $g_x(w)$ was defined in Lemma \ref{lem4.3}. Hence, each
 $G_{m}$ is the mean zero $d$-dimensional standard normal distribution. By Lemma \ref{lem4.3},
 \begin{eqnarray}\label{4.19}
 &\quad\int_{|w|\leq K_{m}}E\big\vert E\big(\exp(i\langle w,V_{m}\rangle)|
 \cG_{m-1}\big)-f_{m}(\Sig_{m-1}w)\big\vert dw\\
 &\leq 2^d\big(C_2\ve^{\wp(1-\ka)}(1-\ve^{3(1-\ka)/4})^{-\wp}\nonumber\\
&+2\phi(\ve^{-(1-\ka)/4})+L\ve^{(1-2\wp)(1-\ka)/4}+\hat Ld\ve^{(3-4\wp)(1-\ka)/4}\big)
  \ve^{-\frac \wp{4}(1-\ka)}(1-\ve^{\frac 34(1-\ka)})^{\frac \wp{4}}\nonumber
 \end{eqnarray}
 where we take $K_{m}=(\ve^{-(1-\ka)}-\ve^{-\frac 14(1-\ka)})^{\wp/4d}=(t_m-s_m)^{\wp/4d}
 \leq(t_m-s_m)^{\wp/2}$. Theorem \ref{thm4.4} requires that $K_m\geq 10^8d$, i.e. in our
 case that $\ve^{-\frac \wp{4d}(1-\ka)}(1-\ve^{\frac 34(1-\ka)})^{\frac \wp{4d}}\geq 10^8d$, which will
 hold true if $\ve\leq\ve_0(\ka,\wp)=(2\cdot 10^{\frac {32d}{\wp}}d^{\frac {4d}\wp})^{-\frac 1{1-\ka}}$
 and taking $\wp=\frac 1{20}$ and $\ka=\frac 35$ we set $\ve_0=(2\cdot 10^{640d}d^{80d})^{-5/2}$.

Next, let $\Psi$ be a $d$-dimensional mean zero standard normal random vector. Then by (\ref{3.4})
and the Chebyshev inequality,
\begin{eqnarray}\label{4.20}
&E\big( G_{m}(\{ y\in\bbR^d:\, |\Sig_{m-1}y|\geq\frac 12(t_m-s_m)^{\frac \wp{4d}}\})\big)\\
&\leq\sup_{y\in\bbR^d}P\{|\sig(y)\Psi| \geq\frac 12(t_m-s_m)^{\frac \wp{4d}}\}\leq 4L^2d
(t_m-s_m)^{-\frac \wp{2d}}\nonumber\\
&=4L^2d\ve^{\frac \wp{2d}(1-\ka)}(1-\ve^{\frac 34(1-\ka)})^{-\frac \wp{2d}}.\nonumber
\end{eqnarray}

Now, Theorem \ref{thm4.4} provides us with random vectors $\{ W_m,\, m\geq 1\}$ satisfying the properties (i)--(iii), in particular, the random $d$-vector $W_{m}$ has the mean zero standard
 normal distribution and it is independent
 of $\cG_{m-1}$ and of $W_1,...,W_{m-1}$ while in view of (\ref{4.19}) and (\ref{4.20}) the
 property (iii) holds true with
 \begin{eqnarray}\label{4.21}
 &\vr_{m}=\vr_m(\ve)=4\wp \ve^{\frac \wp {4d}(1-\ka)}(1-\ve^{\frac
 34(1-\ka)})^{-\frac \wp{4d}}\log(\ve^{-(1-\ka)}-\ve^{-\frac 14(1-\ka)})\\
 &+2\big(C_2\ve^{\wp(1-\ka)}(1-\ve^{\frac 34(1-\ka)})^{-\wp}+2\phi(\ve^{-\frac 14(1-\ka)})
 \nonumber\\
 &+L\ve^{(1-2\wp)(1-\ka)/4}+\hat Ld\ve^{(3-4\wp)(1-\ka)/4}\big)^{1/2}\ve^{-\frac \wp {4}(1-\ka)}
 (1-\ve^{\frac 3{4}(1-\ka)})^{\frac \wp 4}\nonumber\\
 &+4L\sqrt d \ve^{\frac \wp{4d}(1-\ka)}(1-\ve^{\frac 3{4}(1-\ka)})^{-\frac \wp{4d}}\leq
 C_4\ve^{\frac \wp{5d}(1-\ka)}\nonumber
 \end{eqnarray}
 where
 \[
 C_4=8\frac \wp d(1-\ka)\sup_{1\geq\ve>0}(\ve^{\frac \wp{20d}(1-\ka)}\log(1/\ve))+2\sqrt 2(C_2+D)^{1/2}
 +(8L+\sqrt{\hat L})\sqrt d.
 \]

Set $\hat W_k=\sig(X^\ve_{k-2})W_k$.
 As a crucial corollary of Theorem \ref{thm4.4} we will obtain next a uniform $L^{2M}$-bound
 on the difference between the sums of $(t_{k}-t_{k-1})^{1/2}V_k$'s and of $(t_k-t_{k-1})^{1/2}
 \hat W_{k}$'s. Set
 \[
 I(t)=\sum_{k:\, t_k\leq t}(t_k-t_{k-1})^{1/2}(V_k-\hat W_k).
 \]
 \begin{lemma}\label{lem4.5}
 For any positive $\ve\leq\ve_0(\ka,\wp)=(2\cdot 10^{\frac {32d}{\wp}}d^{\frac {4d}\wp})^{-\frac
  1{1-\ka}}$ and an integer $M\geq 1$,
 \begin{equation}\label{4.22}
 E\max_{0\leq t\leq T/\ve^2}|I(t)|^{2M}\leq C_5(M)\ve^{-2M+\frac \wp{10d}(1-\ka)}
 \end{equation}
 where $\wp=\frac 1{20}$, $\frac 12<\ka\leq\frac 35$, $C_5(M)>0$ is given at the end of the
 proof, $D$ is from (\ref{2.2}) and $C_6(M)$ appears in (\ref{4.30}) below.
 \end{lemma}
 \begin{proof}
 First, observe that $I(t)$ changes only at $t=t_{k}$, i.e. it takes on only finitely many
 values for $t\leq T/\ve^2$, and so we can take the maximum in (\ref{4.22}) in place of the
 supremum. The proof of (\ref{4.22}) will rely on Lemmas \ref{lem3.4} and \ref{lem4.1}, so we
 will estimate first $E(I(T/\ve^2))^{2M}$. To do this we have to estimate
 \begin{eqnarray}\label{4.23}
 &A_{2M}=\sup_{n\leq T\ve^{-(1+\ka)}}\sum_{k:\, k\geq n,\, t_{k}\leq T/\ve^2}\\
 &\big((t_{k}-t_{k-1})^{1/2}\| E(V_{k}-\hat W_k|\cG_n\vee\sig\{ U_1,...,U_n\})\|_{2M}\nonumber
 \end{eqnarray}
 taking into account that $V_k$ is $\cG_k=\cF_{-\infty,t_{k}}$-measurable and
 $\hat W_{k}$ is $\cG_k\vee\sig\{ U_1,...,U_k\}$-measurable. First, assume that $k>n$.
 Since $W_k$ is independent of $\cG_{k-1}\vee\sig\{ U_1,...,U_{k-1}\}$ we obtain that
 \begin{eqnarray}\label{4.24}
 &E(\hat W_k|\cG_n\vee\sig\{ U_1,...,U_n\})\\
 &=E\big(\sig(X^\ve_{k-2})E( W_k|\cG_{k-1}\vee\sig\{ U_1,...,U_{k-1}\})|\cG_n\vee\sig\{ U_1,...,U_n\}\big)\nonumber\\
 &=E\big(\sig(X^\ve_{k-2}))EW_k=0.\nonumber
 \end{eqnarray}
  Next, since $V_k$ is independent of
 $\sig\{ U_1,...,U_n\}$ and the latter $\sig$-algebra is independent of $\cG_n$ we obtain that
 (see, for instance, \cite{Chu}, p. 323 or \cite{Ki07}, Remark 4.3),
 \begin{equation}\label{4.25}
 E(V_k|\cG_n\vee\sig\{ U_1,...,U_n\})=E(E(V_k|\cG_{n\vee (k-2)})|\cG_n).
 \end{equation}
 When $k=n$ then, clearly, $E(V_k|\cG_{n\vee (k-2)})=V_k$. When $k=n+1$ then by Lemma \ref{lem3.1},
 \begin{eqnarray}\label{4.26}
 &|E(V_k|\cG_{n\vee (k-2)})|\\
 &\leq (t_{k}-t_{k-1})^{-1/2}\big(|\int_{s_{k}}^{t_{k}}
 E(B(X^\ve_{k-2},\xi(u))|\cF_{-\infty,t_{k-1}})du|+L(s_k-t_{k-1})\big)\nonumber\\
 &\leq L(t_{k}-t_{k-1})^{-1/2}\int_{s_{k}}^{t_{k}}\phi(u-t_{k-1})du+L(t_k-t_{k-1})^{-1/4}.\nonumber
 \end{eqnarray}
 Finally, when $k\geq n+2$ we obtain by Lemma \ref{lem3.1} that
 \begin{eqnarray}\label{4.26+}
 &|E(V_k|\cG_{n\vee (k-2)})|\\
 &=(t_{k}-t_{k-1})^{-1/2}|\int_{t_{k-1}}^{t_{k}}
 E(B(X^\ve_{k-2},\xi(u))|\cF_{-\infty,t_{k-2}})du|\nonumber\\
 &\leq L(t_{k}-t_{k-1})^{-1/2}\int_{t_{k-1}}^{t_{k}}\phi(u-t_{k-2})du.\nonumber
 \end{eqnarray}

 Now, in order to bound $A_{2M}$ it remains to consider the case $k=n$, i.e. to
 estimate $\| V_k-\hat W_k\|_{2M}$ and
 then to combine it with (\ref{4.24})--(\ref{4.26+}). By the H\" older inequality for any $n\geq 1$,
 \begin{eqnarray}\label{4.27}
 &E|V_k-\hat W_k|^{2M}=E(|V_k-\hat W_k|^{2M}\bbI_{|V_k-\hat W_k|\leq\vr_k})\\
 &+E(|V_k-\hat W_k|^{2M}\bbI_{|V_k-\hat W_k|>\vr_k})\nonumber\\
 &\leq\vr^{2M}_k+(E|V_k-\hat W_k|^{2Mn})^{1/n}(P\{|V_k-\hat W_k|>\vr_k\}^{\frac {n-1}n}\nonumber\\
 &\leq\vr^{2M}_k+\vr^{\frac {n-1}n}_k2^{2M}((E|V_k|^{2Mn})^{1/n}+(E|\hat W_k|^{2Mn})^{1/n}).\nonumber
 \end{eqnarray}
 By Lemmas \ref{lem3.1} and \ref{lem3.5},
\begin{eqnarray}\label{4.28}
&(E|V_k|^{2Mn})^{1/n}\leq 2^{2M}\big(C_1(Mn)+2\phi(\ve^{-(1-\ka)})L^{2Mn}(t_k-s_k)^{Mn}\\
&+L^{2Mn}(s_k-t_{k-1})^{-Mn/2}\big)^{1/n}\leq 2^{2M}\big((C_1(Mn))^{1/n}\nonumber\\
&+2^{1/n}(\phi(\ve^{-(1-\ka)}))^{1/n}L^{2M}(t_{k}-s_{k})^{M}
+L^{2M}(s_k-t_{k-1})^{-M/2}\big).\nonumber
\end{eqnarray}

Next, let $\sig(x),\, x\in\bbR^d$ be a Lipschitz continuous (which will be needed later on) symmetric
square root of $A(x)$, i.e. (\ref{2.8}) holds true, and let $\Psi$ be a mean zero $d$-dimensional
Gaussian random vector independent of $X^\ve_{k-2}$ and having the covariance matrix equal to the
identity matrix. Then $\sig(X^\ve_{k-2})\Psi$ has the same distribution as $\hat W_k$, and so
\[
E|\hat W_k|^{2Mn}=E|\sig(X^\ve_{k-2})\Psi|^{2Mn}\leq\sup_x|\sig(x)|^{2Mn}E|\Psi|^{2Mn}.
\]
Since by (\ref{3.4}),
\[
\hat Ld|u|^2\geq\langle A(x)u,u\rangle=\langle\sig(x)u,\sig(x)u\rangle=|\sig(x)u|^2
\]
for any vector $u\in\bbR^d$, it follows that $\sup_x|\sig(x)|\leq\sqrt {\hat Ld}$. We have also
\[
E|\Psi|^{2Mn}\leq d^{Mn}(2\pi)^{-1}\int_{-\infty}^\infty x^{2Mn}e^{-x^2/2}dx=d^{Mn}\prod_{j=1}^{Mn}
(2j-1),
\]
and so
\begin{equation}\label{4.29}
E|\hat W_k|^{2Mn}\leq\hat L^{Mn}d^{2Mn}(2Mn)!.
\end{equation}

Now taking $n=2$ we obtain from (\ref{4.23})--(\ref{4.29}) that
\begin{eqnarray}\label{4.30}
&A_{2M}\leq L\int_1^\infty\phi(r)dr+L\ve^{-(1-\ka)/4}+L\phi(\ve^{-(1-\ka)})T\ve^{-(1+\ka)}\\
&+\max_{k:\, t_{k}\leq T/\ve^2}\big((t_{k}-s_{k})^{1/2}\big(\vr_{k}^{2M}+\vr_{k}^{1/2}
2^{2M}(C_1(2M)\nonumber\\
&+\sqrt 2L^{2M}(\phi(\ve^{-(1-\ka)}))^{1/2}(\ve^{-(1-\ka)}-\ve^{-\frac 14(1-\ka)})^M
+\hat L^Md^{2M}((4M)!)^{1/2})\big)^{1/2M}\big)\nonumber\\
&\leq C_6(M)\ve^{-\frac 12(1-\ka)(1-\frac \wp{10dM})}
\nonumber\end{eqnarray}
where
\begin{equation*}
C_6(M)=L(1+D+TD)+\big(C_4+2^{2M}\sqrt {C_4}(C_1(2M)+\sqrt {2D}L^{2M}+\hat L^Md^{2M}\sqrt {(4M)!}))^{\frac 1{2M}}.
\end{equation*}
This together with Lemma \ref{lem3.4} yields,
\begin{eqnarray}\label{4.31}
&E(I(T/\ve^2))^{2M}\leq 3(2M)!(dT\ve^{-(1+\ka)})^M(A_{2M})^{2M}\\
&\leq 3(2M)!(dT)^{M}C_6^{2M}(M)\ve^{-2M+\frac \wp{10d}(1-\ka)}.\nonumber
\end{eqnarray}
Since by (\ref{4.24})--(\ref{4.26}) for any $n\leq T\ve^{-(1+\ka)}$,
\begin{eqnarray}\label{4.33}
&\sum_{k:\, k>n}(t_{k}-s_{k})^{1/2}E(V_{k}-\hat W_{k}|\cG_n\vee\sig\{ U_1,...,U_n\})\\
&\leq L\int_0^\infty\phi(r)dr+LT\phi(\ve^{-(1-\ka)})\ve^{-(1+\ka)}\nonumber\\
&\leq L\int_1^\infty\phi(r)dr+LT\phi(\ve^{-(1-\ka)})\ve^{-(1+\ka)}\leq L(1+T)D\nonumber
\end{eqnarray}
provided $\frac 12<\ka\leq\frac 35$, we obtain (\ref{4.22}) from (\ref{4.31})--(\ref{4.33}) and
Lemma \ref{lem4.1} with
\[
 C_5(M)=2^{2M-1}\big(3(\frac{2M}{2M-1})^{2M}(2M)!d^MT^M((C_6(M))^{2M}+(L(1+T)D)^{2M}\big),
 \]
completing the proof.
 \end{proof}

\subsection{Diffusion approximation}\label{subsec4.3}
Next, let $\sig(x),\, x\in\bbR^d$ be, as above, a Lipschitz continuous symmetric square root of $A(x)$ and let $W(t),\, t\geq 0$ be a $d$-dimensional Brownian motion.
 Then, the sequences of random vectors $\tilde W_k^\ve =W(t_k)-W(t_{k-1})$ and  $(t_{k}-t_{k-1})^{1/2}
W_{k},\, k<T\ve^{-(1+\ka)}$ have the same distributions. It follows that we can redefine the process $\xi(s),\, -\infty<s<\infty$ and the sequence $W_k,\, k<T\ve^{-(1+\ka)}$
preserving their distributions on maybe richer probability space where there exists a
Brownian motion $W(s),\, s\geq 0$ such that
\begin{equation}\label{4.3.1}
\tilde W_k^\ve =(t-t_{k-1})^{1/2}W_k,\,  k<T\ve^{-(1+\ka)}.
\end{equation}
 To justify this let $Q$ be the joint distribution of $\xi(s),\, -\infty<s<T/\ve^2$ and of
 the sequence $(t-t_{k-1})^{1/2}W_k,\,  k<T\ve^{-(1+\ka)}$ and for a Brownian motion $W$
 let $R$ be the joint distribution of the sequence $\tilde W_k^\ve =W(t_k)-W(t_{k-1})$ and
 of the Brownian motion $W$ itself. Since the
 second marginal of $Q$ coincides with the first marginal of $R$ we can rely on Lemma A1
 from \cite{BP} which implies that the process $\xi$ and the Brownian motion $W$ can be
 redefined on the same rich enough probability space so that (\ref{4.3.1}) holds true for the
 corresponding sequences $W_k$ and $\tilde W_k^\ve$ constructed by them. Thus we can
 and will assume from now on that (\ref{4.22}) in Lemma \ref{lem4.5} holds true for
 \[
 \hat W_k =(t_k-t_{k-1})^{-1/2}\sig(X^\ve_{k-2})(W(t_k)-W(t_{k-1})),\, k<T\ve^{-(1+\ka)}
 \]
 and that the increments $W(t_{k})-W(t_{k-1})$ are independent of the $\sig$-algebras $\cG_{k-1}=\cF_{-\infty,t_{k-1}},\, 1\leq k<T\ve^{-(1+\ka)}$.

\begin{remark}\label{*}
If $B(x,\xi)=\sig(x)\xi$ as in Theorem \ref{thm2.3} then we can rely more directly on the
 strong approximation theorem from \cite{BP} (Theorem 1 there). Namely, set
 \[
 \breve V_m=(t_m-t_{m-1})^{-1/2}\int_{t_{m-1}}^{t_m}\xi(u)du\,\,\mbox{or}\,\,
 \breve V_m=(t_m-t_{m-1})^{-1/2}\sum_{t_{m-1}\leq l<t_m}\xi(l)
 \]
 in the continuous or discrete time cases, respectively. The strong approximation
 theorem from \cite{BP} provides us with a sequence of independent standard normal random
 vectors $W_m,\, m\geq 1$ such that $P\{|\breve V_m-W_m|>\vr_m\}<\vr_m$.
  Then, as in Lemma \ref{lem4.5} we will obtain the estimate
 (\ref{4.22}) for $V_k=\sig(X^\ve_{k-2})\breve V_k$ and $\hat W_k=\sig(X^\ve_{k-2})W_k$.
 Now, relying on Lemma A1 from \cite{BP} we argue as there by redefining the process
 $\xi$ and the sequence $W_k$ preserving their distributions
 on a richer probability space where there exists a Brownian motion $W(t),\, t\geq 0$
 such that $W_k=W(t_k)-W(t_{k-1})$ and since $\sig$ is a bounded matrix,
 we obtain the estimate of Lemma \ref{lem4.5} for $V_k=\sig(X^\ve_{k-2})\breve V_k$ and
 $\hat W_k=(t_k-t_{k-1})^{-1/2}\sig(X^\ve_{k-2})(W(t_k)-W(t_{k-1}))$ with $\breve V_k$'s constructed by the redefined process $\xi$.
\end{remark}

Next, recall that the Lipschitz continuity of $\sig$ follows from \cite{Fre} and \cite{SV} and we saw in the proof above that it is uniformly bounded by $\sqrt {\hat Ld}$. The boundedness and uniform Lipschitz continuity of the functions $b$ and $c$ follows from (\ref{2.3}), (\ref{3.9}), (\ref{3.12})
and (\ref{3.14}). Now, using the Brownian motion $W(t),\, t\geq 0$ constructed above we consider the
new Brownian motion $W_\ve(t)=\ve W(t/\ve^2)$ and introduce the diffusion process $\Xi^\ve(t),\,
t\geq 0$ solving the stochastic differential equation (\ref{2.10}) which we write now with $W_\ve$,
\[
d\Xi^\ve(t)=\sig(\Xi^\ve(t))dW_\ve(t)+(b(\Xi^\ve(t))+c(\Xi^\ve(t)))dt,\,\, \Xi^\ve(0)=x_0
\]
and increasing maybe $L$ from (\ref{2.3}) we will denote the uniform boundedness and the
Lipschitz constants of $\sig,\, b$ and $c$ by the same letter $L$. Now, we introduce the
auxiliary process $\hat\Xi^\ve$ with coefficients frozen at times $\ve^2 t_k,\, k
\leq T/\Del(\ve)=T\ve^{-(1+\ka)}$,
\begin{eqnarray*}
&\hat\Xi^\ve(t)=x_0+\sum_{1\leq k\leq k(\ve,t/\ve^2)}\big(\sig(\Xi^\ve(\ve^2t_{k-2}))
(W_\ve(\ve^2t_{k})-W_\ve(\ve^2t_{k-1}))\\
&+\ve^2(b(\Xi^\ve(\ve^2t_{k-2}))+c(\Xi^\ve(\ve^2t_{k-2})))(t_{k}-t_{k-1})\big)
\end{eqnarray*}
where $t_{-1}=t_0=0$ and $k(\ve,s)$ was defined before Lemma \ref{lem4.3}.

\begin{lemma}\label{lem4.6} For all $\ve\in(0,1]$ and any integer $M\geq 1$,
\begin{equation}\label{4.35}
E\max_{0\leq k\leq T/\Del(\ve)}|\Xi^\ve(\ve^2t_k)-\hat\Xi^\ve(\ve^2t_k)|^{2M}
\leq C_7(M)\ve^{M(1+\ka)}
\end{equation}
where $C_7(M)>0$ appears at the end of the proof.
If $\ve\geq 1$ then (\ref{4.35}) will hold true with $\ve^{2M(1+\ka)}$ in place of $\ve^{M(1+\ka)}$.
\end{lemma}
\begin{proof}
First, we write
\begin{eqnarray}\label{4.36}
&E\max_{0\leq k\leq T/\Del(\ve)}|\Xi^\ve(\ve^2t_k)-\hat\Xi^\ve(\ve^2t_k)|^{2M}\\
&\leq 2^{2M-1}(E\max_{0\leq k\leq T/\Del(\ve)}|J_1(\ve^2t_k)|^{2M}+E\max_{0\leq k\leq
 T/\Del(\ve)}|J_2(\ve^2t_k)|^{2M})\nonumber
\end{eqnarray}
where
\[
J_1(t)=\int_{0}^{t}\big(\sig(\Xi^\ve(s))-\sig(\Xi^\ve(([s/\Del(\ve)]-1)\Del(\ve)))\big)dW_\ve(s)
\]
and
\begin{eqnarray*}
&J_2(t)=\int_{0}^{t}\big(b(\Xi^\ve(s))+c(\Xi^\ve(s))-b(\Xi^\ve(([s/\Del(\ve)]-1)\Del(\ve)))\\
&-c(\Xi^\ve(([s/\Del(\ve)]-1)\Del(\ve)))\big)ds.
\end{eqnarray*}

By the standard martingale moment inequalities for stochastic integrals (see, for instance, \cite{IW},
Chapter 3 or \cite{Mao}, Section 1.7),
\begin{eqnarray}\label{4.37}
&E\max_{0\leq k\leq T/\Del(\ve)}|J_1(\ve^2t_k)|^{2M}\leq (\frac {2M}{2M-1})^{2M}(M(2M-1))^MT^{M-1}\\
&\times\int_{0}^{[T/\Del(\ve)]\Del(\ve)}E|\sig(\Xi^\ve(s))
-\sig(\Xi^\ve(([s/\Del(\ve)]-1)\Del(\ve)))|^{2M}ds\nonumber\\
&\leq 2^{2M}M^{3M}(2M-1)^{-M}T^{M-1}L^{2M}\nonumber\\
&\times\sum_{0\leq k\leq T/\Del(\ve)}\int_{\ve^2t_{k-1}}^{\ve^2t_{k}}E|\Xi^\ve(s)-\Xi^\ve(\ve^2t_{k-2})|^{2M}ds.
\nonumber\end{eqnarray}
By (\ref{3.4}) and the Cauchy-Schwarz inequality,
\begin{eqnarray}\label{4.38}
&E\max_{0\leq k\leq T/\Del(\ve)}|J_2(\ve^2t_k)|^{2M}\leq L^{2M}T^{2M-1}(1+16L\int_0^\infty\phi(r)dr)^{2M}\\
&\times\sum_{0\leq k\leq T/\Del(\ve)}\int_{\ve^2t_{k-1}}^{\ve^2t_{k}}
E|\Xi^\ve(s)-\Xi^\ve(\ve^2t_{k-2})|^{2M}ds.\nonumber\\
\nonumber\end{eqnarray}

Again by (\ref{3.4}) and the moment inequalities for stochastic integrals
\begin{eqnarray}\label{4.38+}
&E|\Xi^\ve(s)-\Xi^\ve(\ve^2t_{k-2})|^{2M}\leq 2^{2M-1}\big(E|\int_{\ve^2t_{k-2}}^s\sig(\Xi^\ve(u))dW_\ve(u)|^{2M}\\
&+L^{2M}(1+L)^{2M}(s-\ve^2t_{k-2})^{2M}\big)\leq 2^{2M-1}L^{2M}(s-\ve^2t_{k-2})^M
\nonumber\\
&\times(M^M(2M-1)^M+(1+L)^{2M}(s-t_{k-2})^M).
\nonumber
\end{eqnarray}
Since $s\in[\ve^2t_{k-1},\ve^2t_k]$ here, we have that $s-\ve^2t_{k-2}\leq 2\Del(\ve)$, and
 so (\ref{4.35}) follows from (\ref{4.36})--(\ref{4.38+}) with
 \begin{eqnarray*}
&C_7(M)=2^{6M}L^{4M}T^M\big(2^MM^{3M}+T^M(1+16L\int_0^\infty\phi(r)dr)^{2M}\big)\\
&\times((1+L)^{2M}+M^M(2M-1)^M).
\end{eqnarray*}
 \end{proof}

Next, we define
\[
\hat X^\ve(t)=x_0+\sum_{0\leq k< k(\ve,t/\ve^2)}\big(\ve\al^\ve_k(X^\ve_{k-1})+\ve^2(b(X^\ve_{k-1})
+c(X^\ve_{k-1}))(t_{k+1}-t_k)\big)
\]
where $k(\ve,t)$ was defined at the beginning of Section \ref{sec4} and $\al_k^\ve$ is the same as in Lemma \ref{lem3.3}.
In order to use the estimate of
Lemma \ref{lem3.3} we will need first to compare $\hat X^\ve$ with the sum appearing there.

\begin{lemma}\label{lem4.7} For all $0<\ve\leq 1$,
\begin{equation}\label{4.39}
E\sup_{0\leq t\leq T}|\hat X^\ve(t)-\breve X^\ve(t)|^{2M}\leq C_8(M)\ve^{M\min(\frac 12,\,5-7\ka)}
\end{equation}
where $\breve X^\ve$ is the same as in Lemma \ref{lem3.3} and $C_8(M)>0$ is given at the end of
the proof.
If $\ve\geq 1$ then (\ref{4.39}) will still be true if we replace $\ve^{M\min(\frac 12,\, 5-7\ka)}$ by $\ve^{3M}$.
\end{lemma}
\begin{proof} The left hand side of (\ref{4.39}) equals to
\begin{eqnarray}\label{4.40}
&\ve^{4M}E\sup_{0\leq t\leq T}\big\vert\sum_{0\leq k< k(\ve,t/\ve^2)}((b(X^\ve_{k-1})+c(X^\ve_{k-1}))(t_{k+1}-t_k)\\
&-\be^\ve_k-\gam^\ve_k)\big\vert^{2M}\leq 2^{2M-1}\ve^{4M}(E\sup_{0\leq t\leq T}|I_1(t)|^{2M}+E\sup_{0\leq t\leq T}|I_2(t)|^{2M})
\nonumber\end{eqnarray}
where
\[
I_1(t)=\sum_{0\leq k<  k(\ve,t/\ve^2)}(b(X^\ve_{k-1})(t_{k+1}-t_k)-\be^\ve_k)
\]
and
\[
I_2(t)=\sum_{0\leq k<  k(\ve,t/\ve^2)}(c(X^\ve_{k-1})(t_{k+1}-t_k)-\gam^\ve_k).
\]

Next, we will estimate $I_1$ and $I_2$ and the reader should bear in mind that in view of
(\ref{4.40}), in order to obtain (\ref{4.39}) we will have to multiply these estimates by
the appropriate power of $\ve$. By Lemma \ref{lem3.1} for any $k\geq n+1$,
\begin{eqnarray}\label{4.41}
&\big\vert E(b(X^\ve_{k-1})(t_{k+1}-t_k)-\be^\ve_k)|\cF_{-\infty,t_n})\big\vert\\
&=\big\vert E\big(\int_{t_k}^{t_{k+1}}E(b(X^\ve_{k-1})-b(X^\ve_{k-1},\xi(u))|\cF_{-\infty,t_{k-1}})du|
\cF_{-\infty,t_n}\big)\big\vert\nonumber\\
&\leq 4L\ve^{-(1-\ka)}\phi(\ve^{-(1-\ka)}).
\nonumber\end{eqnarray}
When $n=k+1$ then $b(X^\ve_{k-1})(t_{k+1}-t_k)-\be^\ve_k$ is $\cF_{-\infty,t_n}$-measurable
and we estimate then, and also when $n=k$, the left hand side of (\ref{4.41}) just by $2L\ve^{-(1-\ka)}$. Thus, relying on
 Lemmas \ref{lem3.4} and \ref{lem4.1} considered with $\eta_k=b(X^\ve_{k-1})(t_{k+1}-t_k)
 -\be_k^\ve$, $\cG_k=\cH_k=\cF_{-\infty,t_k}$ and taking into account that
 $k(\ve,T/\ve^2)\leq T\ve^{-(1+\ka)}$ we obtain
\begin{eqnarray}\label{4.42}
&E\sup_{0\leq t\leq T}|I_1(t)|^{2M}
\leq 2^{6M}L^{2M}d^M(\ve^{-(1-\ka)}\\
&+T\ve^{-2}\phi(\ve^{-(1-\ka)}))^{2M}
(3(\frac {2M}{2M-1})^{2M}(2M)!(T\ve^{-(1+\ka)})^M+1).\nonumber
\end{eqnarray}

Next,
\begin{eqnarray}\label{4.43}
&E\sup_{0\leq t\leq T}|I_2(t)|^{2M}\leq 3^{2M-1}(E\sup_{0\leq t\leq T}|I_{21}(t)|^{2M}\\
&+E\sup_{0\leq t\leq T}|I_{22}(t)|^{2M}+E\sup_{0\leq t\leq T}|I_{23}(t)|^{2M})\nonumber
\end{eqnarray}
where
\[
I_{21}(t)=\sum_{0\leq k<  k(\ve,t/\ve^2)}(c(X^\ve_{k-1})(t_{k+1}-t_k)-\int_{t_k}^{t_{k+1}}du
\int_{t_{k-1}}^uc(X^\ve_{k-1},u,v)dv),
\]
\[
I_{22}(t)=\sum_{0\leq k<k(\ve,t/\ve^2)}\int_{t_k}^{t_{k+1}}du\int_{t_{k}}^u(\gam(X^\ve_{k-1},u,v)
-c(X^\ve_{k-1},u,v))dv,
\]
\[
I_{23}(t)=\sum_{0\leq k< k(\ve,t/\ve^2)}\int_{t_k}^{t_{k+1}}du\int_{t_{k-1}}^{t_k}
(\gam(X^\ve_{k-1},u,v)-c(X^\ve_{k-1},u,v))dv
\]
and we set $\gam(x,u,v)=\nabla_xB(x,\xi(u))B(x,\xi(v))$.

By Lemma \ref{lem3.2},
\begin{eqnarray}\label{4.44}
&\sup_{0\leq t\leq T}|I_{21}(t)|\leq 2L^2T\ve^{-(1+\ka)}\int_0^{\ve^{-(1-\ka)}}du
\int^\infty_{u+\ve^{-(1-\ka)}}\phi(r)dr\\
&\leq 2L^2T\ve^{-2}\int^\infty_{\ve^{-(1-\ka)}}\phi(r)dr.\nonumber
\end{eqnarray}
The second term in the right hand side of (\ref{4.43}) we estimate exactly as in (\ref{4.42}).
Namely, by Lemma \ref{lem3.1} for any $k\geq n+1$ similarly to (\ref{4.41}),
\begin{eqnarray*}
&|E\big(\int_{t_k}^{t_{k+1}}du\int_{t_{k}}^u(\gam(X^\ve_{k-1},u,v)-c(X^\ve_{k-1},u,v))dv
|\cF_{-\infty,t_n}\big)|\\
&\leq 4L^2\ve^{-2(1-\ka)}\phi(\ve^{-(1-\ka)}),
\end{eqnarray*}
while for $k=n$ and $k=n-1$ we estimate this conditional expectation just by $2L^2\ve^{-2(1-\ka)}$.
Hence, by Lemmas \ref{lem3.4} and \ref{lem4.1},
\begin{eqnarray}\label{4.45}
&\quad E\sup_{0\leq t\leq T}|I_{22}(t)|^{2M}\leq 2^{6M}L^{4M}d^M(\ve^{-2(1-\ka)}+T\ve^{-(3-\ka)}\phi(\ve^{-(1-\ka)}))^{2M}\\
&\times(3(\frac {2M}{2M-1})^{2M}(2M)!(T\ve^{-(1+\ka)})^M+1).\nonumber
\end{eqnarray}

In order to estimate the last term in the right hand side of (\ref{4.43}) we observe that by Lemma \ref{lem3.1} for any $k\geq n+1$,
\begin{eqnarray*}
&\big\vert E\big(\int_{t_k}^{t_{k+1}}du\int_{t_{k-1}}^{t_k}(\gam(X^\ve_{k-1},u,v)-
c(X^\ve_{k-1},u,v))dv|\cF_{-\infty,t_n}\big)\big\vert\\
&\leq\big\vert E\big(\int_{t_k}^{t_{k+1}}du\int_{t_{k-1}+\frac 12\ve^{-(1-\ka)}}^{t_k}E(\gam(X^\ve_{k-1},u,v)\\
&-c(X^\ve_{k-1},u,v)|\cF_{-\infty,t_{k-1}})dv|
\cF_{-\infty,t_n}\big)\big\vert\\
&+\big\vert E\big(\int_{t_k}^{t_{k+1}}du\int^{t_{k-1}+\frac 12\ve^{-(1-\ka)}}_{t_{k-1}}E(\gam(X^\ve_{k-1},u,v)-c(X^\ve_{k-1},u,v)|\cF_{-\infty,v})dv|
\cF_{-\infty,t_n}\big)\big\vert\\
&\leq 4L^2\ve^{-2(1-\ka)}\phi(\frac 12\ve^{-(1-\ka)})
\end{eqnarray*}
where we use also (\ref{1.2}), (\ref{2.3}), (\ref{3.5}) and (\ref{3.6}). Applying Lemmas \ref{lem3.4}
and \ref{lem4.1} we obtain from here similarly to (\ref{4.45}) that
\begin{eqnarray}\label{4.46}
&E\sup_{0\leq t\leq T}|I_{23}(t)|^{2M}\leq 2^{6M}L^{4M}d^M(1+T\ve^{(-3+\ka)}
\phi(\frac 12\ve^{-(1-\ka)}))^{2M}\\
&\times(3(\frac {2M}{2M-1})^{2M}(2M)!(T\ve^{-(1+\ka)})^M+1)\nonumber
\end{eqnarray}
which together with (\ref{4.40}) and (\ref{4.42})--(\ref{4.45}) yields (\ref{4.39}) with
\[
C_8(M)=2^{6M+4}(L+1)^{2M}d^MD^{2M}(T+1)^{3M}(\frac {2M}{2M-1})^{2M}(2M)!,
\]
completing the proof of the lemma.
\end{proof}

\subsection{Completing the proof of Theorem \ref{thm2.1}}\label{subsec4.4}
 Denote
\begin{eqnarray*}
&\tilde\Xi^\ve(t)=x_0+\sum_{0\leq k< k(\ve,t/\ve^2)}(\sig(X^\ve_{k-1})(W_\ve(\ve^2t_{k+1})-W_\ve(\ve^2t_{k}))\\
&+\ve^2(b(X^\ve_{k-1})+c(X^\ve_{k-1}))(t_{k+1}-t_k).
\end{eqnarray*}
Then
\begin{eqnarray}\label{4.47}
& E\sup_{0\leq s\leq T}|\hat X^\ve(s)-\hat\Xi^\ve(s)|^{2M}=E\max_{0\leq k<k(\ve,T/\ve^2)}|
\hat X^\ve(\ve^2t_k)\\
&-\hat\Xi^\ve(\ve^2t_k)|^{2M}\leq 2^{2M-1}(E\max_{0\leq k<k(\ve,T/\ve^2)}|\hat X^\ve(\ve^2t_k)-\tilde\Xi^\ve(\ve^2t_k)|^{2M}
\nonumber\\
&+E\max_{0\leq k<k(\ve,T/\ve^2)}|\tilde\Xi^\ve(\ve^2t_k)-\hat\Xi^\ve(\ve^2t_k)|^{2M})\nonumber.
\end{eqnarray}
By Lemmas \ref{lem4.1} and \ref{lem4.5} for any $\ve\in(0,\ve_0(\ka,\wp)]$,
\begin{eqnarray}\label{4.48}
&E\max_{0\leq k\leq n}|\hat X^\ve(\ve^2t_k)-\tilde\Xi^\ve(\ve^2t_k)|^{2M}=E\max_{0\leq k\leq n}\\
&\big\vert\sum_{0\leq l\leq k}\big(\ve\int_{t_l}^{t_{l+1}}B(X^\ve_{l-1},\xi(u))du
-\sig(X^\ve_{l-1})(W_\ve(\ve^2t_{l+1})-W_\ve(\ve^2t_{l}))\big)\big\vert^{2M}\nonumber\\
&\leq\ve^{2M}E\sup_{0\leq t\leq T}|I(t)|^{2M}\leq C_5(M)\ve^{\frac \wp{10d}(1-\ka)}.\nonumber
\end{eqnarray}

In order to estimate the second term in the right hand side of (\ref{4.47}) introduce the
$\sig$-algebras $\cQ_s=\cF_{-\infty,s}\vee\sig\{ W(u),\, 0\leq u\leq s\}$ and observe that
by our construction for each $k$ the increment $W(t_{k+1})-W(t_{k})$ is independent of $\cQ_{t_k}$.
 On the other hand, for any $k\geq n$ both $X^\ve(\ve^2t_k)$ and $\Xi^\ve(\ve^2t_k)$ are
 $\cQ_{t_k}$-measurable. Hence,
 \begin{eqnarray*}
 &\cI_1(t_k)=\sum_{0\leq l\leq k-1}\big(\sig(X^\ve(\ve^2t_{l-1}))-\sig(\Xi^\ve(\ve^2t_{l-1}))\big)
 (W_\ve(\ve^2t_{l+1})-W_\ve(\ve^2t_{l}))\\
 &=\int_0^{\ve^2t_k}(\sum_{l=1}^k\bbI_{[s_l,t_l]}(s/\ve^2))\big(\sig(X^\ve(([\frac s{\Del(\ve)}]-1)\Del(\ve))-\sig(\Xi^\ve
 (([\frac s{\Del(\ve)}]-1)\Del(\ve))\big)dW_\ve(s)\nonumber
 \end{eqnarray*}
 can be viewed as a stochastic integral, and so by the moment martingale estimates for stochastic integrals
 (see, for instance, \cite{IW} or \cite{Mao}),
 \begin{eqnarray}\label{4.49}
 &E\max_{1\leq k\leq n}|\cI_1(t_k)|^{2M}\leq (\frac {2M}{2M-1})^{2M} E|\cI_1(t_n)|^{2M}\\
 &\leq(\frac {2M}{2M-1})^{2M}(M(2M-1))^M\ve^{2(M-1)}t_n^{M-1}\nonumber\\
 &\times\int_0^{\ve t_n}E|\sig(X^\ve
 (([\frac s{\Del(\ve)}]-1)\Del(\ve))-\sig(\Xi^\ve(([\frac s{\Del(\ve)}]-1)\Del(\ve))|^{2M}ds\nonumber\\
 &\leq (\frac {2M}{2M-1})^{2M}(M(2M-1))^ML^{2M}T^{M-1}\Del(\ve)\nonumber\\
 &\times\sum_{0\leq k<n}E|X^\ve(\ve^2t_{k-1})-\Xi^\ve(\ve^2t_{k-1})|^{2M}.
\nonumber \end{eqnarray}
A similar estimate can be obtained relying on Lemmas \ref{lem3.4} and \ref{lem4.1} instead of moment
inequalities for stochastic integrals as above.

Next, observe that
\begin{eqnarray}\label{4.50}
&E\max_{0\leq k\leq T\ve^{-(1+\ka)}}|\tilde\Xi^\ve(\ve^2t_k)-\hat\Xi^\ve(\ve^2t_k)|^{2M}\\
&\leq 2^{2M-1}(E\max_{0\leq k\leq T\ve^{-(1+\ka)}}|\cI_1(t_k)|^{2M}+E\max_{0\leq k\leq T\ve^{-(1+\ka)}}|\cI_2(t_k)|^{2M}
\nonumber\\
&+E\max_{0\leq k\leq T\ve^{-(1+\ka)}}|\cI_3(t_k)|^{2M})
\nonumber\end{eqnarray}
where
\begin{eqnarray*}
&\cI_2(t_k)=\ve^2\sum_{0\leq l\leq k-1}\big(b(X^\ve(\ve^2t_{l-1}))+c(X^\ve(\ve^2t_{l-1}))\\
&-b(\Xi^\ve(\ve^2t_{l-1}))-c(\Xi^\ve(\ve^2t_{l-1}))\big)(t_{l+1}-t_l).
\end{eqnarray*}
By (\ref{3.4}) we have
\begin{eqnarray}\label{4.51}
&|\cI_2(t_k)|^{2M}\leq L^{2M}(16L\int_0^\infty\phi(r)dr+1)^{2M}(\Del(\ve))^{2M}\\
&\times(\sum_{0\leq l\leq k-1}|X^\ve(\ve^2t_{l-1})-\Xi^\ve(\ve^2t_{l-1})|)^{2M}\leq L^{2M}(16L(D+1)+1)^{2M}\nonumber\\
&\times(\Del(\ve))^{2M}k^{2M-1}\sum_{0\leq l\leq k-1}
|X^\ve(\ve^2t_{l-1})-\Xi^\ve(\ve^2t_{l-1})|^{2M}\nonumber\\
&\leq L^{2M}(16L(D+1)+1)^{2M}T^{2M-1}\Del(\ve)\sum_{0\leq l\leq k-1}
|X^\ve(\ve^2t_{l-1})-\Xi^\ve(\ve^2t_{l-1})|^{2M}.\nonumber
\end{eqnarray}
Since $\cI_3(t_k)$ is a stochastic integral we can rely on the corresponding martingale moment inequalities
(see, for instance, Section 1.7 in \cite{Mao}) which yields
\begin{eqnarray}\label{4.52}
&E\max_{0\leq k\leq n}|\cI_3(t_k)|^{2M}\leq (\frac {2M}{2M-1})^{2M} E|\cI_3(t_n)|^{2M}\\
&\leq (\frac {2M}{2M-1})^{2M}(M(2M-1))^ML^{2M}T^M\ve^{3(1-\ka)/4}=\hat C(M)\ve^{3(1-\ka)/4}.\nonumber
\end{eqnarray}

Now denote
\[
G^\ve_k=E\max_{0\leq l\leq k}|X^\ve(\ve^2t_l)-\Xi^\ve(\ve^2t_l)|^{2M}.
\]
Then we obtain from (\ref{3.15}), (\ref{4.35}), (\ref{4.39}) and (\ref{4.47})--(\ref{4.52}) that
for $n\leq T/\Del(\ve)=T\ve^{-(1+\ka)},\, 0<\ve\leq\ve_0(\ka,\wp)$,
\begin{equation}\label{4.53}
G^\ve_n\leq C_9(M)\ve^{\min(2\ka-1,5-7\ka,\frac \wp{10d}(1-\ka))}+C_{10}(M)\Del(\ve)\sum_{0\leq k\leq n-1}G^\ve_k
\end{equation}
where
\begin{eqnarray*}
&C_9(M)=4^{2M}((2L^2T(2L+1)+4L)^{2M}+3C_3(M)+3C_5(M)\\
&+C_7(M)+C_8(M))\,\,\mbox{and}\,\,\, C_{10}=L^{2M}T^{M-1}(\frac {2^{2M}M^{3M}}{(2M-1)^M}+(16L(D+1)+1)^{2M}T^M).
\end{eqnarray*}
By the discrete (time) Gronwall inequality (see, for instance, \cite{Cla}),
\begin{equation}\label{4.54}
G^\ve_{k(\ve,T/\ve^2)}\leq C_9(M)\ve^{\min(2\ka-1,5-7\ka,\frac \wp{10d}(1-\ka))}\exp(C_{10}(M)T).
\end{equation}

It remains to estimate deviations of our continuous time processes within intervals of time
$(\ve^2t_k,\ve^2t_{k+1})$ which where not taken into account in previous estimates,
i.e. we have to deal now with
\begin{eqnarray*}
&\cJ_1=E\sup_{0\leq t\leq T}|X^\ve(t)-X^\ve(t_{k(\ve,t/\ve^2)})|^{2M}\\
&\mbox{and}\,\,\, \cJ_2=E\sup_{0\leq t\leq T}|\Xi^\ve(t)-\Xi^\ve(t_{k(\ve,t/\ve^2)})|^{2M}.
\end{eqnarray*}
By the straightforward estimates using (\ref{1.1}) and (\ref{2.3}) we obtain
\begin{equation}\label{4.55}
\cJ_1\leq (\frac {2L}{\ve}\Del(\ve))^{2M}=(2L)^{2M}\ve^{2M\ka}
\end{equation}
and
\begin{equation}\label{4.56}
\cJ_2\leq 2^{2M-1}(\cJ_3+(2L)^{2M}(\Del(\ve))^{2M})
\end{equation}
where
\[
\cJ_3=E\max_{0\leq k\leq T/\Del(\ve)}\sup_{0\leq s\leq\Del(\ve)}|\int_{\ve^2t_k}^{\ve^2t_k+s}
\sig(\Xi^\ve(u))dW_\ve(u)|^{2M}.
\]

By the Jensen (or Cauchy-Schwarz) inequality and the uniform moment estimates for stochastic
integrals
\begin{eqnarray}\label{4.57}
&\cJ_3\leq\big( E\max_{0\leq k\leq T/\Del(\ve)}\sup_{0\leq s\leq\Del(\ve)}|\int_{\ve^2t_k}^{\ve^2t_k+s}
\sig(\Xi^\ve(u))dW_\ve(u)|^{4M}\big)^{1/2}\\
&\leq\big( \sum_{0\leq k\leq T/\Del(\ve)}E\sup_{0\leq s\leq\Del(\ve)}|\int_{\ve^2t_k}^{\ve^2t_k+s}
\sig(\Xi^\ve(u))dW_\ve(u)|^{4M}\big)^{1/2}\nonumber\\
&\leq(\frac {4M}{4M-1})^{2M}\big( \sum_{0\leq k\leq T/\Del(\ve)}E|\int_{\ve^2t_k}^{\ve^2t_{k+1}}
\sig(\Xi^\ve(u))dW_\ve(u)|^{4M}\big)^{1/2}\nonumber\\
&\leq 2^{5M}M^{3M}L^{2M}(4M-1)^{-M}\Del(\ve)^{M-\frac 12}T^{1/2}\nonumber
\end{eqnarray}
since $|\sig(x)|\leq\sqrt L$. Combining (\ref{4.53})--(\ref{4.57}) we complete the proof of Theorem \ref{thm2.1} assuming that $\wp=\frac 1{20},\,\ka=\frac 35$ and $\ve\in(0,\ve_0]$.
\qed

\section{Discrete time case  }\label{sec5}\setcounter{equation}{0}

We start with the discrete time version of Lemma \ref{3.2}.
\begin{lemma}\label{lem5.1}
The limits (\ref{2.12}) and (\ref{2.13}) exist uniformly in $\im$ and for all integers
$\im,n\geq 0$,
 \begin{equation}\label{5.1}
 |nc(x)-\sum_{l=\im}^{\im+n}\sum_{m=\im-n}^lc(x,l,m)|\leq 2L^2\sum_{l=0}^n\sum_{m=n+l}^\infty\phi(r)
\end{equation}
and
 \begin{equation}\label{5.2}
 |na_{jk}(x)-\sum_{l=\im}^{\im+n}\sum_{m=\im}^{\im+n}a_{jk}(x,l,m)|
 \leq 2L^2\sum_{l=0}^n\sum_{m=n+l}^\infty\phi(r).
\end{equation}
  Moreover, $c(x)$ and $b(x)=Eb(x,\xi(0))$ are once and $a_{jk}(x)$ is twice differentiable
  for $j,k=1,...,d$ with bounds given by (\ref{3.4}).
  \end{lemma}
  \begin{proof} The proof is the same as in Lemma \ref{lem3.2} just by replacing integrals in
  time there by the corresponding sums.
  \end{proof}

  Next, we set again $t_k=t_k(\ve)=k\ve^{-(1-\ka)}=k\Del(\ve)\ve^{-2}$ and $X^\ve_k=X_d^\ve(\Del(\ve)k)=X_d^\ve(\ve^2t_k)$ where $X^\ve_d(t)$ is defined for all $t\leq T/\ve^2$
  by (\ref{2.14}). Define also
  \begin{eqnarray*}
  &\al_k^\ve(x)=\sum_{t_k\leq l<t_{k+1}}B(x,\xi(l)),\,\be_k^\ve(x)=\sum_{t_k\leq l<t_{k+1}}b(x,\xi(l)),\\
  &\gam_k^\ve(x)=\sum_{t_k\leq l<t_k}\sum_{t_{k-1}\leq m<l}\frac {\partial B(x,\xi(l))}{\partial x}
  B(x,\xi(m))
  \end{eqnarray*}
  and set $\al_k^\ve=\al_k^\ve(X^\ve_{k-1})$, $\be^\ve_k=\be^\ve_k(X^\ve_{k-1})$ and
   $\gam_k^\ve=\gam_k^\ve(X^\ve_{k-1})$. We set again
   \[
  \breve X_d^\ve(t)=\sum_{k=0}^{[t/\Del(\ve)]-1}(\ve\al_k^\ve+\ve^2\be^\ve_k+\ve^2\gam_k^\ve).
  \]
   and obtain
   \begin{lemma}\label{lem5.2}
   For any $T\geq t>s\geq 0$,
    \begin{eqnarray}\label{5.3}
 &\big\vert X_d^\ve(t)-X_d^\ve(s)-\breve X_d^\ve(t)+\breve X_d^\ve(s)\big\vert\\
& \leq L^2T\ve^{2\ka-1}(1+\ve)\big(\frac 76L(1+\ve)+\frac 32\ve^{1-\ka}(1+L(1+2\ve)\big)+2L\ve^{\ka}(1+\ve).\nonumber
 \end{eqnarray}
 \end{lemma}
 \begin{proof}
 Using the Taylor two terms expansion we have
  \begin{eqnarray}\label{5.4}
 &X_d^\ve(\Del(\ve)(k+1))-X_d^\ve(\Del(\ve)k)\\
 &=\ve\sum_{t_k\leq l<t_{k+1}}\big(B(X_d^\ve(\ve^2l),\xi(l))+\ve b(X_d^\ve(\ve^2l),\xi(l))\big)
 \nonumber\\
 &=\ve\sum_{t_k\leq l<t_{k+1}}\big(B(X^\ve_{k-1},\xi(l))+\ve b(X^\ve_{k-1},\xi(l))\big)\nonumber\\
 &+\ve\sum_{t_k\leq l<t_{k+1}}\nabla_xB(X^\ve_{k-1},\xi(l))(X_d^\ve(\ve^2l)-X^\ve_{k-1})+\ve R^\ve_{1,k} \nonumber\\
 &=\ve\al_k^\ve+\ve^2(\be_k^\ve+\gam^\ve_k)+\ve R^\ve_k\nonumber
 \end{eqnarray}
 since for $l\geq t_k$,
 \begin{equation}\label{5.5}
 X_d^\ve(\ve^2l)-X^\ve_{k-1}=\ve\sum_{t_{k-1}\leq j<l}B(X^\ve_{k-1},\xi(j))+\ve R^\ve_{2,k,l},
 \end{equation}
 and the errors $R^\ve_{1,k},\, R^\ve_{2,k}=\sum_{t_k\leq l<t_{k+1}}\nabla_xB(X^\ve_{k-1},\xi(l))
 R^\ve_{2,k,l}$  and $R^\ve_{k}=R^\ve_{1,k}+\ve R^\ve_{2,k}$ of the corresponding
 Taylor expansions are estimated in the same way as in Lemma \ref{lem3.3} (replacing integrals by
 sums). Summing in $k$ and estimating $|X^\ve_d(u)-X^\ve_d([u/\Del(\ve)]\Del(\ve))|$ by $L\ve(1+\ve)$ we obtain (\ref{5.3}).
 \end{proof}

 We have also
  \begin{lemma}\label{lem5.3}
  For any $t\geq 0$, $x\in\bbR^d$ and an integer $M\geq 1$,
  \begin{equation}\label{5.6}
  E|\sum_{0\leq l<t}B(x,\xi(l))|^{2M}\leq C_1(M)t^M.
  \end{equation}
  \end{lemma}
  \begin{proof} In fact, we can take in (\ref{5.6}) even a bit smaller $C_1(M)$ than in
  Lemma \ref{lem3.5} since there is no need here to approximate the integral by a sum.
  The proof is the same as in Lemma \ref{lem3.5} relying on Lemma \ref{lem3.4}.
  \end{proof}

  Next, for any integer $n>0$ and $x\in\bbR^d$ introduce the characteristic function
 \[
 f_n(x,w)=E\exp(i\langle w,n^{-1/2}\sum_{0\leq l<n}B(x,\xi(l))\rangle),\,\, w\in\bbR^d.
 \]
 \begin{lemma}\label{lem5.4} For any integer $n>0$ and $x\in\cR^d$,
 \begin{equation}\label{5.7}
 |f_n(x,w)-\exp(-\frac 12\langle A(x)w,w\rangle)|\leq C_2n^{-\wp}
 \end{equation}
 for all $w\in\bbR^d$ with $|w|\leq n^{\wp/2}$ where $\wp$ and $C_2$ can be taken as in
 Lemma \ref{lem3.6}.
 \end{lemma}
 \begin{proof} The proof is by the block-gap technique and it proceeds in the same way as in Lemma \ref{3.6} and in \cite{DP}.
 \end{proof}

 The remaining part of the proof of Theorem \ref{thm2.2} goes on exactly as in Section \ref{sec4} replacing any integral of the form $\int_s^tB(x,\xi(u))du$, $\int_s^tb(x,\xi(u))du$, $\int_s^tdu
 \int_\tau^uc(x,u,v)dv$,  $\int_s^tdu\int_\tau^u\gam(x,u,v)dv$ and $\int_s^t\phi(r)dr$ there by
 the sums $\sum_{s\leq l<t}B(x,\xi(l))$, $\sum_{s\leq l<t}b(x,\xi(l))$, $\sum_{s\leq l<t}
 \sum_{\tau\leq m<l}c(x,l,m)$,  $\sum_{s\leq l<t}\sum_{\tau\leq m<l}\gam(x,l,m)$ and
  $\sum_{s\leq l<t}\phi(l)$, respectively, and taking into account that most of the proof in Section \ref{sec4} is for sequences and sums of random vectors, and so it is well adapted to the discrete
  time case.   \qed

  \section{Computing Dynkin games values  }\label{sec6}\setcounter{equation}{0}

Set $n_k=[t_k]=[k\ve^{-(1-\ka)}]$ and let $\cT^{\Del}$ be the set of all stopping times with respect
to the filtration $\cF_{-\infty,n_k},\, k\geq 0$ taking on values $n_k,\, k=0,1,...,k_{\max}$
 where $k_{\max}=
[T/\Del(\ve)]$ if $n_{[T/\Del(\ve)]}=T/\ve^2$ and $k_{\max}=[T/\Del(\ve)]+1$ and $n_{k_{\max}}=T/\ve^2$
if $n_{[T/\Del(\ve)]}<T/\ve^2$. Denote by $\cQ_{n_k}$ the $\sig$-algebra $\cF_{-\infty,n_k}\vee\sig\{ U_i,\, 1\leq i\leq k\}$ where, recall, $U_1,U_2,...$ is a sequence of i.i.d. uniformly distributed
random variables appearing in Theorem \ref{thm4.4} (which should be applied now for the discrete time
setup). Let $\cT^\cQ$ be the
 set of all stopping times with respect to the filtration $\cQ_{n_k},\, k\geq 0$ taking on values
 $n_k,\, k=0,1,...,k_{\max}$. Next, introduce the payoffs based on $\breve X^\ve_d$ (the same as
 in Lemma \ref{lem5.2}),
 \[
 \breve R^\ve(s,t)=G_s(\breve X_d)\bbI_{s<t}+F_t(\breve X_d)\bbI_{t\leq s}
 \]
 and the game values corresponding to sets of stopping times $\cT^\Del$ and $\cT^\cQ$,
 \[
 V^\ve_\Del=\inf_{\sig\in\cT^\Del}\sup_{\tau\in\cT^\Del}ER^\ve(\ve^2\sig,\ve^2\tau),
 \]
 \[
 \breve V^\ve_\Del=\inf_{\sig\in\cT^\Del}\sup_{\tau\in\cT^\Del}E\breve R^\ve(\ve^2\sig,\ve^2\tau),
 \]
 \[
 \mbox{and}\,\,\,\breve V^\ve_\cQ=\inf_{\sig\in\cT^\cQ}\sup_{\tau\in\cT^\cQ}E
 \breve R^\ve(\ve^2\sig,\ve^2\tau).
 \]

 \begin{lemma}\label{lem6.1} For all $\ve\in (0,1]$,
 \begin{equation}\label{6.1}
 |V^\ve-V^\ve_\Del|\leq\ve^{\ka}(K(1+|x|)+2KL+L),
 \end{equation}
 where $x=X^\ve(0)$, and
  \begin{equation}\label{6.2}
 |V_\Del^\ve-\breve V^\ve_\Del|\leq 2L^2(2L+1)T\ve^{2\ka-1}+4L\ve^{\ka}.
 \end{equation}
 \end{lemma}
 \begin{proof} For any $\zeta\in\cT_{0N_\ve}$ set $\zeta^\Del=\min\{ n_k:\, n_k\geq\zeta\}$
 which defines a stopping time from $\cT^\Del$ satisfying
 \begin{equation}\label{6.3}
 \ve^2\zeta+\Del(\ve)\geq\ve^2\zeta^\Del\geq\ve^2\zeta.
 \end{equation}
 Since $\cT_{0N_\ve}\supset\cT^\Del$ we see that
 \[
 V^\ve\geq\inf_{\zeta\in\cT_{0N_\ve}}\sup_{\eta\in\cT^\Del}ER^\ve(\ve^2\zeta,\ve^2\eta).
 \]
 Then for any $\vt>0$ there exists $\zeta_\vt\in\cT_{0N_\ve}$ such that
 \[
 V^\ve\geq\sup_{\eta\in\cT^\Del}ER^\ve(\ve^2\zeta_\vt,\ve^2\eta)-\vt,
 \]
 and so
 \begin{eqnarray}\label{6.4}
 &V^\ve\geq\sup_{\eta\in\cT^\Del}ER^\ve(\ve^2\zeta_\vt^\Del,\ve^2\eta)-\vt\\
 &-\sup_{\eta\in\cT^\Del}E(R^\ve(\ve^2\zeta_\vt^\Del,\ve^2\eta)-R^\ve(\ve^2\zeta_\vt,\ve^2\eta)\nonumber\\
 &\geq V^\ve_\Del-\vt-\sup_{\eta\in\cT^\Del}J_1^\ve(\ve^2\zeta_\vt,\ve^2\eta)\nonumber
 \end{eqnarray}
 where for any $\zeta\in\cT_{0N_\ve}$ and $\eta\in\cT^\Del$,
 \[
 J_1^\ve(\ve^2\zeta,\ve^2\eta)=E(R^\ve(\ve^2\zeta^\Del,\ve^2\eta)-R^\ve(\ve^2\zeta,\ve^2\eta)).
 \]

 Since $\zeta^\Del\geq\zeta$,
 \[
 R^\ve(\ve^2\zeta,\ve^2\eta)=G_{\ve^2\zeta}(X^\ve_d)\,\,\mbox{whenever}\,\, R^\ve(\ve^2\zeta^\Del,
 \ve^2\eta)=G_{\ve^2\zeta^\Del}(X^\ve_d).
 \]
 Hence, by (\ref{2.23}) and (\ref{6.3}),
 \begin{eqnarray}\label{6.5}
& R^\ve(\ve^2\zeta^\Del,\ve^2\eta)-R^\ve(\ve^2\zeta,\ve^2\eta)\leq\max\big( |G_{\ve^2\zeta^\Del}(X^\ve_d)
-G_{\ve^2\zeta}(X^\ve_d)|,\\
& |F_{\ve^2\zeta^\Del}(X^\ve_d)-F_{\ve^2\zeta}(X^\ve_d)|\big)\leq K\big(\Del(\ve)(1+|x|+
\ve\sum_{0\leq l\leq[T/\ve^2]}(|\sig(X^\ve_d(l\ve^2))\xi(l)|\nonumber\\
&+\ve|b(X^\ve_d(l\ve^2)|))+\ve\max_{0\leq k\leq k_{\max}}\max_{1\leq l\leq\ve^{-(1-\ka)}}\nonumber\\
&|\sum_{n_k+l\leq j\leq n_{k+1}}\sig(X_d^\ve(j\ve^2)\xi(j)|\big)
\leq K\Del(\ve)(1+|x|)+KL(1+\ve)\ve^\ka+L\ve^\ka.\nonumber
\end{eqnarray}
Taking here $\zeta_\vt$ in place of $\zeta$ we obtain from (\ref{6.4}) and (\ref{6.5}) that
\[
V^\ve\geq V^\ve_\Del-\vt-\ve^\ka(K\ve(1+|x|)+KL(1+\ve)+L)
\]
and since $\vt>0$ is arbitrary and $\ve$ does not depend on $\vt$ we have that
\begin{equation}\label{6.6}
V^\ve\geq V^\ve_\Del-\ve^\ka(K\ve(1+|x|)+KL(1+\ve)+L).
\end{equation}

On the other hand, since the Dynkin game here has a value (see, for instance, \cite{Ki20}, Section
6.2.2) we can write also that
\begin{equation}\label{6.7}
V^\ve=\sup_{\eta\in\cT_{0N_\ve}}\inf_{\zeta\in\cT_{0N_\ve}}ER^\ve(\ve^2\zeta,\ve^2\eta)\leq
\inf_{\zeta\in\cT^\Del}ER^\ve(\ve^2\zeta,\ve^2\eta_\vt)+\vt
\end{equation}
for each $\vt>0$ and some $\eta_\vt\in\cT_{0N_\ve}$. Introducing $\eta_\vt^\Del$ and arguing as
above we obtain that
\[
V^\ve\leq V^\ve_\Del+\ve^\ka(K\ve(1+|x|)+KL(1+\ve)+L)
\]
which together with (\ref{6.6}) completes the proof of (\ref{6.1}).

In order to prove (\ref{6.2}) we observe that by (\ref{2.22}) and Lemma \ref{lem5.2},
\begin{eqnarray}\label{6.8}
&|V^\ve_\Del-\breve V^\ve_\Del|\leq\sup_{\zeta\in\cT^\Del}\sup_{\eta\in\cT^\Del}E|R^\ve(\ve^2\zeta,
\ve^2\eta)-\breve R^\ve(\ve^2\zeta,\ve^2\eta)|\\
&\leq\max\big(E\sup_{0\leq t\leq T}|F_t(X^\ve_d)-F_t(\breve X^\ve_d)|,\,|G_t(X^\ve_d)-
G_t(\breve X^\ve_d)|\big)\nonumber\\
&\leq KE\sup_{0\leq t\leq T}|X^\ve_d(t)-\breve X^\ve_d(t)|\leq 2L^2(2L+1)T\ve^{2\ka-1}+4L\ve^\kappa\nonumber
\end{eqnarray}
yielding (\ref{6.2}).
  \end{proof}

\begin{lemma}\label{lem6.2} For all $\ve>0$,
\begin{equation}\label{6.9}
\breve V^\ve_\Del=\breve V^\ve_\cQ.
\end{equation}
\end{lemma}
\begin{proof}
We prove (\ref{6.9}) obtaining both $\breve V^\ve_\Del$ and $\breve V^\ve_\cQ$ by the standard dynamical
programming (backward recursion) procedure (see, for instance, Section 1.3.2 in \cite{Ki20}).
Namely, we have $\breve V^\ve_\Del=\breve V^\ve_{\Del,0}$ and $\breve V^\ve_\cQ=\breve
V^\ve_{\cQ,0}$ where
\begin{equation}\label{6.10}
\breve V^\ve_{\Del,k_{\max}}=F_T(\breve X^\ve)=\breve V^\ve_{\cQ,k_{\max}}
\end{equation}
proceeding recursively
\[
\breve V^\ve_{\Del,k}=\min\big(G_{\ve^2n_k}(\breve X^\ve),\,\max(F_{\ve^2n_k}(\breve X^\ve),\,
E(\breve V^\ve_{\Del,k+1}|\cF_{-\infty,n_k}))\big)
\]
and
\[
\breve V^\ve_{\cQ,k}=\min\big(G_{\ve^2n_k}(\breve X^\ve),\,\max(F_{\ve^2n_k}(\breve X^\ve),\,
E(\breve V^\ve_{\cQ,k+1}|\cF_{-\infty,n_k}))\big).
\]

Since each $\sig$-algebra $\sig\{ U_1,...,U_k\}$ is independent of $\xi_1,\xi_2,...$ by the
construction, i.e. it is independent of all $\sig$-algebras $\cF_{-\infty,l},\, l=0,\pm 1,...$,
and so it is independent of $X^\ve_d$,  it follows (see, for instance, \cite{Chu}, p.323
or \cite{Ki07}, Remark 4.3) that
\[
E(\breve V^\ve_{\Del,k+1}|\cF_{-\infty,n_k})=E(\breve V^\ve_{\Del,k+1}|\cQ_{n_k}),
\]
and so starting from (\ref{6.10}) we proceed recursively to $\breve V^\ve_{\Del,0}=
\breve V^\ve_{\cQ,0}$
proving (\ref{6.9}).
\end{proof}

Next, we turn our attention to the diffusion $\Xi$ constructed in Theorem \ref{thm2.2} and
consider the corresponding Dynkin game value $V^\Xi$ given by (\ref{2.21}). Set
\[
\hat\cG^\Xi_{n_k}=\sig\{ W_\ve(\ve^2n_l)-W_\ve(\ve^2n_{l-1}):\, l=1,...,k\}
\subset\cG^\Xi_{n_k}=\sig\{ W_\ve(\ve^2u):\, u\leq n_k\}
\]
 and observe that by the construction
\begin{equation}\label{6.11}
\hat\cG^\Xi_{n_k}\subset\cQ_{n_k}=\cF_{-\infty,n_k}\vee\sig\{ U_i,\, 1\leq i\leq k\}
\end{equation}
where $W_\ve$ is the Brownian motion which emerges in the proof of Theorem \ref{thm2.2} in the same
way as in Section \ref{sec4}. Let $\cT_\Del^\Xi$ be the set of all stopping times with respect to
the filtration $\cG^\Xi_{n_k},\, k\geq 0$ and $\cT_{\Del}^\cQ$ be the set of all stopping times
with respect to the filtration $\cQ_{n_k},\, k\geq 0$, both taking values $n_k$ when $k$ runs from 0 to $k_{\max}$. Set
\[
\hat R^\Xi(s,t)=G_s(\hat\Xi^\ve)\bbI_{s<t}+F_t(\hat\Xi^\ve)\bbI_{t\leq s}
\]
where, similarly to Section \ref{sec4},
\begin{eqnarray*}
&\hat\Xi^\ve(t)=\sum_{0\leq k\leq k(\ve,t/\ve^2)}\big(\sig(\Xi^\ve(\ve^2n_{k-1})(W_\ve(\ve^2n_{k+1})-
W_\ve(\ve^2n_k))\\
&+\ve^2b(\Xi^\ve(\ve^2n_{k-1}))(n_{k+1}-n_k)\big).
\end{eqnarray*}
Set
 \[
 V^\Xi_\Del=\inf_{\zeta\in\cT^\Xi_\Del}\sup_{\eta\in\cT^\Xi_\Del}ER^\Xi(\ve^2\zeta,\ve^2\eta),
 \]
 \[
 \hat V^\Xi_\Del=\inf_{\zeta\in\cT^\Xi_\Del}\sup_{\eta\in\cT^\Xi_\Del}E
 \hat R^\Xi(\ve^2\zeta,\ve^2\eta)
 \]
 and
 \[
 \hat V^\Xi_\cQ=\inf_{\zeta\in\cT^\cQ}\sup_{\eta\in\cT^\cQ}E\hat R^\Xi(\ve^2\zeta,\ve^2\eta).
 \]

\begin{lemma}\label{lem6.3} For any $\ve\in(0,1]$,
\begin{equation}\label{6.12}
|V^\Xi-V^\Xi_\Del|\leq K\Del(\ve)(|x|+L\sqrt T(1+\sqrt T))+12KT^{1/4}L\sqrt {\Del(\ve)},
\end{equation}
where $x=\Xi^\ve(0)$, and
\begin{equation}\label{6.13}
|V^\Xi_\Del-\hat V^\Xi_\Del|\leq K\sqrt {C_7(2)}\sqrt {\Del(\ve)}.
\end{equation}
\end{lemma}
\begin{proof}
The proof is similar to Lemma \ref{lem6.1} but here in place of estimates for $X_d^\ve$ we have
to use moment estimates for diffusions. Set $\cT_{0T}^{\Xi,\ve}=\{\zeta:\,\ve^2\zeta\in\cT_{0T}^\Xi\}$
where, recall, $\cT^\Xi_{0T}$ is the set of stopping times with respect to the filtration $\cF_t^\Xi=
\sig\{ W_\ve(s),\, s\leq t\}$ having values in $[0,T]$. For any $\xi\in\cT_{0T}^{\Xi,\ve}$ define
$\zeta^\Del=\min\{ n_k:\, n_k\geq\zeta\}$ which yields a stopping time from $\cT^\Xi_\Del$ satisfying
(\ref{6.3}). Since $\cT_\Del^\Xi\subset\cT_{0,T}^{\Xi,\ve}$ we have that
\[
V^\Xi\geq\inf_{\zeta\in\cT_{0T}^{\Xi,\ve}}\sup_{\eta\in\cT_{0T}^{\Xi,\ve}}ER^\Xi(\ve^2\zeta,\ve^2\eta).
\]
In the same way as in (\ref{6.4}) we obtain that for some $\zeta_\vt\in\cT_{0T}^{\Xi,\ve}$,
\begin{equation}\label{6.14}
V^\Xi\geq V^\Xi_\Del-\vt-\sup_{\eta\in\cT_\Del^\Xi}J_2^\ve(\ve^2\zeta_\vt,\ve^2\eta)
\end{equation}
where for any $\zeta\in\cT_{0T}^{\Xi,\ve}$ and $\eta\in\cT_{\Del}^\Xi$,
\[
J^\ve_2(\ve^2\zeta,\ve^2\eta)=E(R^\Xi(\ve^2\zeta^\Del,\ve^2\eta)-R^\Xi(\ve^2\zeta,\ve^2\eta)).
\]

As in (\ref{6.5}) we obtain from (\ref{2.23}) and (\ref{6.3}) that
\begin{eqnarray}\label{6.15}
&R^\Xi(\ve^2\zeta^\Del,\ve^2\eta)-R^\Xi(\ve^2\zeta,\ve^2\eta)\leq K\big(\Del(\ve)(1+\sup_{0\leq t\leq T}|\Xi^\ve(t)|\\
&+\max_{0\leq k\leq k_{\max}}\sup_{\ve^2n_k\leq s\leq\ve^2n_{k+1}}|\Xi^\ve(\ve^2n_{k+1})-\Xi^\ve(s)|\big).\nonumber
\end{eqnarray}
By the moment estimates for stochastic integrals (see, for instance, Ch.3 in \cite{IW} or \cite{Mao},
 Section 1.7) and the Cauchy-Schwarz inequality,
\begin{eqnarray}\label{6.16}
&E\sup_{0\leq t\leq T}|\Xi^\ve(t)|\leq |x|+E\sup_{0\leq t\leq T}|\int_0^t\sig(\Xi^\ve(u))dW_\ve(u)|\\
&+\int_0^TE|b(\Xi^\ve(u))|du\leq |x|+L\sqrt T(1+\sqrt T),\nonumber
\end{eqnarray}
recalling that $\sup_x|\sig(x)|\leq L$.

Next, we write
\begin{eqnarray}\label{6.17}
&E\max_{1\leq k\leq k_{\max}}\sup_{\ve^2n_{k+1}\geq s\geq\ve^2n_k}|\Xi^\ve(\ve^2n_{k+1})-\Xi^\ve(s)|\\
&\leq(\sum_{1\leq k\leq k_{\max}}E\sup_{\ve^2n_{k+1}\geq s\geq\ve^2n_k}|\Xi^\ve(\ve^2n_{k+1})-\Xi^\ve(s)|^4)^{1/4}\nonumber
\end{eqnarray}
and
\begin{eqnarray}\label{6.18}
&E\sup_{\ve^2n_{k+1}\geq s\geq\ve^2n_k}|\Xi^\ve(\ve^2n_{k+1})-\Xi^\ve(s)|^4\\
&\leq 8E|\Xi^\ve(\ve^2n_{k+1})-\Xi^\ve(\ve^2n_k)|^4+8E\sup_{\ve^2n_{k+1}\geq s\geq\ve^2n_k}|\Xi^\ve(s)-\Xi^\ve(\ve^2n_{k})|^4.\nonumber
\end{eqnarray}

Again, by the standard moment estimates for stochastic integrals
\begin{eqnarray}\label{6.19}
&E|\Xi^\ve(\ve^2n_{k+1})-\Xi^\ve(\ve^2n_k)|^4\leq 8E|\int_{\ve^2n_k}^{\ve^2n_{k+1}}\sig(\Xi^\ve(u))dW_\ve(u)|^4\\
&+8E(\int_{\ve^2n_k}^{\ve^2n_{k+1}}b(\Xi^\ve(u))du)^4\leq 288\Del(\ve)\int_{\ve^2n_k}^{\ve^2n_{k+1}}E|\sig(\Xi^\ve(u))|^4du\nonumber\\
&+8L^4(\Del(\ve))^4\leq 8L^4(\Del(\ve))^2(36+(\Del(\ve))^2)\nonumber
\end{eqnarray}
and
\begin{eqnarray}\label{6.20}
&E\sup_{\ve^2n_{k+1}\geq s\geq\ve^2n_k}|\Xi^\ve(s)-\Xi^\ve(\ve^2n_{k})|^4\\
&\leq 8(4/3)^4E|\int_{\ve^2n_k}^{\ve^2n_{k+1}}\sig(\Xi^\ve(u))dW_\ve(u)|^4\nonumber\\
&+8E(\int_{\ve^2n_k}^{\ve^2n_{k+1}}b(\Xi^\ve(u))du)^4\leq 8L^4(\Del(\ve))^2(36(4/3)^4+(\Del(\ve))^2).\nonumber
\end{eqnarray}

Combining (\ref{6.14})--(\ref{6.20}) we obtain the required lower bound for $V^\Xi-V^\Xi_\Del$
taking into account that $\vt>0$ is arbitrary. On the other hand, since the Dynkin game has a
value under our conditions (see, for instance, \cite{Ki20}, Section 6.2.2) we can write that
\[
V^\Xi=\sup_{\eta\in\cT_{0T}^{\Xi,\ve}}\inf_{\zeta\in\cT_{0T}^{\Xi,\ve}}ER^\Xi(\ve^2\zeta,\ve^2\eta)
\leq\inf_{\zeta\in\cT_{\Del}^{\Xi}}ER^\Xi(\ve^2\zeta,\ve^2\eta_\vt)+\vt
\]
for any $\vt>0$ and some $\eta_\vt\in\cT_{0T}^{\Xi,\ve}$. Introducing $\eta_\vt^\Del$ and relying
on the same arguments as above we obtain the corresponding upper bound for $V^\Xi-V^\Xi_\Del$ and
complete the proof of (\ref{6.12}).

Next, we obtain (\ref{6.13}) by (\ref{2.22}), Lemma \ref{lem4.6} and the Cauchy-Schwarz inequality
\begin{eqnarray}\label{6.21}
&|V^\Xi_\Del-\hat V^\Xi_\Del|\leq\sup_{\zeta\in\cT_\Del^\Xi}\sup_{\eta\in\cT_\Del^\Xi}
E|R^\Xi(\ve^2\zeta,\ve^2\eta)-\hat R^\Xi(\ve^2\zeta,\ve^2\eta)|\\
&\leq KE\sup_{0\leq t\leq T}(\Xi^\ve(t)-\hat\Xi^\ve(t)|\leq K\sqrt {C_7(2)}\sqrt {\Del(\ve)}\nonumber
\end{eqnarray}
completing the proof of the lemma.
\end{proof}

Next, we introduce the new process $\Psi^\ve$, first recursively at the times $\ve^2n_k$ and then
extending it for all $t\in[0,T]$ in the piece-wise constant fashion. Namely, we set $\Psi^\ve(0)=x_0$ and (with $n_0=n_{-1}=0$),
\begin{eqnarray*}
&\Psi^\ve(\ve^2n_{k+1})=\Psi^\ve(\ve^2n_k)+\sig(\Psi^\ve(\ve^2n_{k-1}))(W_\ve(\ve^2n_{k+1})-W_\ve(\ve^2n_k))\\
&+\ve^2b(\Psi^\ve(\ve^2n_{k-1}))(n_{k+1}-n_k)
\end{eqnarray*}
for $k=0,1,...,k_{\max}-1$. Set also $\Psi^\ve(t)=\Psi^\ve(\ve^2n_k)$ if $\ve^2n_k\leq t<\ve^2n_{k+1}$. The reader should not mix the process $\Psi^\ve$ with the process
$\hat\Xi^\ve$ appearing before Lemma \ref{lem6.3} as $\Psi^\ve$ is defined recursively
through itself while $\hat\Xi^\ve$ is defined through $\Xi^\ve$.

\begin{lemma}\label{lem6.4} For any $\ve\in(0,1]$,
\begin{equation}\label{6.22}
E\max_{0\leq k\leq k_{\max}}|\Xi^\ve(\ve^2n_k)-\Psi^\ve(\ve^2n_k)|^2\leq 3C_7(2)\ve^{1+\ka}\exp(24L^2d(T+1)).
\end{equation}
\end{lemma}
\begin{proof} We have
\begin{eqnarray*}
&|\Xi^\ve(\ve^2n_k)-\Psi^\ve(\ve^2n_k)|^2\leq 3\big(|\Xi^\ve(\ve^2n_k)-\hat\Xi^\ve(\ve^2n_k)|^2\\
&+\big\vert\sum_{0\leq l<k}(\sig(\Xi^\ve(\ve^2n_{l-1}))-\sig(\Psi^\ve(\ve^2n_{l-1})))(W_\ve(\ve^2n_{l+1})
-W_\ve(\ve^2n_{l}))\big\vert^2\\
&+(\ve^2\sum_{0\leq l<k}|b(\Xi^\ve(\ve^2n_{l-1}))-b(\Psi^\ve(\ve^2n_{l-1}))|(n_{l+1}-n_l))^2\big),
\end{eqnarray*}
and so
\begin{eqnarray}\label{6.23}
&\quad\quad\max_{0\leq k\leq n}|\Xi^\ve(\ve^2n_k)-\Psi^\ve(\ve^2n_k)|^2\leq 3\big(\max_{0\leq k\leq n}|\Xi^\ve(\ve^2n_k)-\hat\Xi^\ve(\ve^2n_k)|^2\\
&+\max_{0\leq k\leq n}|M_k|^2+4k_{\max}(\Del(\ve))^2\sum_{0\leq l<n}|b(\Xi^\ve(\ve^2n_{l-1}))-
b(\Psi^\ve(\ve^2n_{l-1}))|^2\nonumber
\end{eqnarray}
where
\[
M_k=\sum_{0\leq l<k}(\sig(\Xi^\ve(\ve^2n_{l-1}))-\sig(\Psi^\ve(\ve^2n_{l-1})))(W_\ve(\ve^2n_{l+1})
-W_\ve(\ve^2n_{l}))
\]
is a martingale with respect to the filtration $\{ \cG^\Xi_{n_k},\, k\geq 0\}$ since
 $\sig(\Xi^\ve(\ve^2n_{l-1}))-\sig(\Psi^\ve(\ve^2n_{l-1}))$ is $\cG_{n_{l-1}}^\Xi$-measurable while
$W_\ve(\ve^2n_{l+1})-W_\ve(\ve^2n_{l})$ is independent of $\cG_{n_l}^\Xi\supset\cG_{n_{l-1}}^\Xi$.

Hence, by the Doob martingale moment inequality and by the Lipschitz continuity of $\sig$ (with
the constant $L$),
\begin{equation}\label{6.24}
E\max_{0\leq k\leq n}|M_k|^2\leq 4E|M_n|^2\leq 4L^2d\ve^2\sum_{0\leq k\leq n}Q^\ve_k(n_{k+1}-n_k)
\end{equation}
where
\[
Q^\ve_n=E\max_{0\leq k\leq n}|\Xi^\ve(\ve^2n_k)-\Psi^\ve(\ve^2n_k)|^2.
\]
By (\ref{4.35}) considered with $n_k=[t_k]$ in place of $t_k$ which yields the same, by (\ref{6.23})
and (\ref{6.24}) we obtain that
\[
Q^\ve_n\leq 3C_7(2)\ve^{1+\ka}+24L^2d\Del(\ve)\sum_{0\leq k<n}Q^\ve_k.
\]
Thus, by the discrete (time) Gronwall inequality (see \cite{Cla}),
\[
Q^\ve_n\leq 3C_7(2)\ve^{1+\ka}\exp(24L^2d\Del(\ve)n)
\]
and since $n\leq [T/\Del(\ve)]+1$, (\ref{6.22}) follows.
\end{proof}

Next, we introduce the values of Dynkin games with payoffs based on the process $\Psi^\ve$. Namely,
we set
\[
R^\Psi(s,t)=G_s(\Psi^\ve)\bbI_{s<t}+F_t(\Psi^\ve)\bbI_{t\leq s},
\]
 \[
 V^\Psi_\Del=\inf_{\zeta\in\cT^\Xi_\Del}\sup_{\eta\in\cT^\Xi_\Del}ER^\Psi(\ve^2\zeta,\ve^2\eta)
 \]
 \[
\mbox{and}\,\, V^\Psi_\cQ=\inf_{\zeta\in\cT^\cQ}\sup_{\eta\in\cT^\cQ}
ER^\Psi(\ve^2\zeta,\ve^2\eta).
 \]
\begin{lemma}\label{lem6.5} For any $\ve>0$,
\begin{equation}\label{6.25}
V^\Psi_\Del=V^\Psi_\cQ.
\end{equation}
\end{lemma}
\begin{proof}
As in Lemma \ref{lem6.2} we will prove (\ref{6.25}) obtaining both $V_\Del^\Psi$ and $V^\Psi_\cQ$ by
the dynamical programming procedure. Again, we have $V^\Psi_\Del=V^\Psi_{\Del,0}$ and   $V^\Psi_\cQ=
V^\Psi_{\cQ,0}$ where $V^\Psi_{\Del,k_{\max}}=F_T(\Psi^\ve)=V^\Psi_{\cQ,k_{\max}}$ and for
$k=k_{\max}-1, k_{\max}-2,...,0$,
\[
V^\Psi_{\Del,k}=\min\big(G_{\ve^2n_k}(\Psi^\ve),\,\max(F_{\ve^2n_k}(\Psi^\ve),\, E(V^\Psi_{\Del,k+1}|
\cG_{n_k}^\Xi))\big)
\]
and
\[
V^\Psi_{\cQ,k}=\min\big(G_{\ve^2n_k}(\Psi^\ve),\,\max(F_{\ve^2n_k}(\Psi^\ve),\, E(V^\Psi_{\cQ,k+1}|
\cQ_{n_k}))\big).
\]

For any vectors $x_0,x_1,x_2,...,x_{k_{\max}}\in\bbR^d$ set $x(0)=x_0$, $x(t)=x_k$ if $\ve^2n_k\leq t
<\ve^2n_{k+1}$ and define the functions
\[
q_{k(\ve,t/\ve^2)}(x_1,...,x_{k(\ve,t/\ve^2)})=F_t(x)\,\,\mbox{and}\,\,
r_{k(\ve,t/\ve^2)}(x_1,...,x_{k(\ve,t/\ve^2)})=G_t(x).
\]
Introduce
\[
\Phi_l(x_1,...,x_l)=\min\big(r_l(x_1,...,x_l),\,\max(q_l(x_1,...,x_l),\, h(x_1,...,x_l))\big)
\]
where
\[
h(x_1,...,x_l)=E\Phi_{l+1}\big(x_1,...,x_l,\, x_l+\sig(x_{l-1})(W_\ve(\ve^2n_{l+1})-W_\ve(\ve^2n_l))\big).
\]
Since by (\ref{6.11}) and the construction
$\Psi^\ve(\ve^2n_l)$ is both $\cF_{-\infty,n_l}$ and $\cQ_{n_l}$-measurable while
$W_\ve(\ve^2n_{l+1})-W_\ve(\ve^2n_l)$ is independent of both $\cF_{-\infty,n_l}$ and $\cQ_{n_l}$ 
we see by induction that
\[
V^\Psi_{\cQ,l}=\Phi_l(\Psi^\ve(\ve^2n_1),\Psi^\ve(\ve^2n_2),...,\Psi^\ve(\ve^2n_l))=V^\Psi_{\Del,l},
\]
for all $l=k_{\max},k_{\max}-1,...,0$ where $\Phi_0=\min(F_0(x_0),\max(G_0(x_0),E\Phi_1(x_0+
\sig(x_0)W_\ve(\ve^2n_1)))$,
and (\ref{6.25}) follows.
\end{proof}

Now we can complete the proof of Theorem \ref{thm2.3} writing first,
\begin{eqnarray}\label{6.26}
&|V^\Xi-V^\ve|\leq |V^\ve-V^\ve_\Del|+|V^\ve_\Del-\breve V^\ve_\cQ|+|\breve V^\ve_\cQ-V^\Psi_\cQ|\\
&+|V^\Psi_\cQ-\hat V^\Xi_\Del|+|\hat V^\Xi_\Del-V^\Xi_\Del| +|V^\Xi_\Del-V^\Xi|.\nonumber
\end{eqnarray}
 It remains to estimate $|\breve V^\ve_\cQ-V^\Psi_\cQ|$ and $|V^\Psi_\cQ-\hat V^\Xi_\Del|=|V^\Psi_\Del-\hat V^\Xi_\Del|$ since all other terms in the right hand side of (\ref{6.26}) are dealt with by Lemmas
 \ref{lem6.1}--\ref{lem6.3}. In both remaining estimates we use the fact that the game values there are
 defined with respect to the same sets of stopping times which will allow us to rely on uniform bounds
 on distances between the corresponding processes. By (\ref{2.22}),
 \begin{eqnarray}\label{6.27}
 &|\breve V^\ve_\cQ-V^\Psi_\cQ|\leq\sup_{\zeta\in\cT^{\cQ}}\sup_{\eta\in\cT^{\cQ}}E|\breve R^\ve(\ve^2\zeta,\ve^2\eta)- R^\Psi(\ve^2\zeta,\ve^2\eta)|\\
 &\leq\max(E\sup_{0\leq t\leq T}|F_t(\breve X^\ve_d)-F_t(\Psi^\ve)|,\,
 E\sup_{0\leq t\leq T}|G_t(\breve X^\ve_d)-F_t(\Psi^\ve)|)\nonumber\\
 &\leq KE\sup_{0\leq t\leq T}|\breve X^\ve_d(t)-\Psi^\ve(t)|=KE\max_{0\leq k\leq k_{\max}}
 |\breve X^\ve_d(\ve^2n_k)-\Psi^\ve(\ve^2n_k)|.\nonumber
 \end{eqnarray}
 Next, by Lemmas \ref{lem5.2}, \ref{lem6.4} and Theorem \ref{thm2.2},
  \begin{eqnarray}\label{6.28}
&E\max_{0\leq k\leq k_{\max}}|\breve X^\ve_d(\ve^2n_k)-\Psi^\ve(\ve^2n_k)|\\
&\leq  E\max_{0\leq k\leq k_{\max}}|\breve X^\ve_d(\ve^2n_k)-X^\ve_d(\ve^2n_k)|\nonumber\\
&+E\max_{0\leq k\leq k_{\max}}|X^\ve_d(\ve^2n_k)-\Xi(\ve^2n_k)|+E\max_{0\leq k\leq k_{\max}}
|\Xi(\ve^2n_k)-\Psi^\ve(\ve^2n_k)|\nonumber\\
&\leq 2L^2(2L+1)T\ve^{2\ka-1}+\sqrt {C_0(2)}\ve^{\del/2}+\sqrt {3C_7(2)}\exp(12L^2d(T+1))
\ve^{\frac 12(1+\ka)}. \nonumber
\end{eqnarray}

Similarly, by (\ref{2.22}) and by Lemmas \ref{lem4.6} and \ref{lem6.4},
\begin{eqnarray}\label{6.29}
&|V^\Psi_\Del-\hat V^\Xi_\Del|\leq\sup_{\zeta\in\cT^{\Del}}\sup_{\eta\in\cT^{\Del}}E|R^\Psi(\ve^2\zeta,\ve^2\eta)\\
&- \hat R^\Xi(\ve^2\zeta,\ve^2\eta)|\leq KE\max_{0\leq k\leq k_{\max}}|\Psi^\ve(\ve^2n_k)
-\hat\Xi^\ve(\ve^2n_k)|\nonumber\\
&\leq KE\max_{0\leq k\leq k_{\max}}|\Psi^\ve(\ve^2n_k)-\Xi^\ve(\ve^2n_k)|+KE\max_{0\leq k\leq k_{\max}}|\Xi^\ve(\ve^2n_k)\nonumber\\
&-\hat\Xi^\ve(\ve^2n_k)|\leq K\ve^{\frac 12(1+\ka)}(\sqrt {C_7(2)}+\sqrt {3C_7(2)}\exp(12L^2d(T+1))).\nonumber
\end{eqnarray}
Combining (\ref{6.26}) together with (\ref{6.27})--(\ref{6.29}) and Lemmas \ref{lem6.1}--\ref{lem6.3}
we complete the proof of Theorem \ref{thm2.3}.
\qed


\begin{thebibliography}{Bow75}

\itemsep=\smallskipamount



\bibitem{Bil} P. Billingsley, {\em Convergence of Probability Measures}, 2nd ed., J.Willey,
New York, 1999.


\bibitem{Bor} A.N. Borodin, {\em A limit theorem for solutions of differential equations with
random right-hand side}, Theory Probab. Appl. 22 (1977), 482--497.


\bibitem{BF} A.N. Borodin and M.I. Freidlin, {\em Fast oscillating random perturbations of dynamical
systems with conservation laws}, Annales de l'I.H.P., sec. B, 31 (1995), 485--525.


\bibitem{BDG} E. Bayraktar, Ya. Dolinsky and J. Guo, {\em Recombining tree approximations
for optimal stopping for diffusions}, SIAM J. Financial Math. 9 (2018), 602--633.



\bibitem{Bra} R.C. Bradley, {\em Introduction to Strong Mixing Conditions,}
Kendrick Press, Heber City, 2007.


\bibitem{BP} I. Berkes and W. Philipp, {\em Approximation theorems for independent and
weakly dependent random vectors}, Annals Probab. 7 (1979), 29--54.


\bibitem{Chu} K.-L. Chung, {\em A Course in Probability}, 3d edition, Acad. Press,
San Diego, Ca., 2001.


\bibitem{Cla} D. S. Clark, {\em A short proof of a discrete Gronwall inequality},
Discrete Appl. Math. 16 (1987), 279--281.



\bibitem{CE} R. Cogburn and J.A. Ellison, {\em A stochastic theory of adiabatic invariance},
Commun. Math. Phys. 149 (1992), 97--126.


\bibitem{CFKMZ} I. Chevyrev, P.K. Friz, A. Korepanov, I. Melbourne and H. Zhang,
{\em Deterministic homogenization under optimal moment assumptions for fast-slow
 systems}, Part 2, arXiv: 1903.10418, 2020.


\bibitem{Dol} Y. Dolinsky, {\em Applications of weak convergence for hedging of game options},
Ann. Appl. Probab. 20 (2010), 1891--1906.


\bibitem{Dud} R.M. Dudley, {\em Real Analysis and Probability}, Cambridge Univ. Press, New York,
2002.


\bibitem{DP} H. Dehling and W. Philipp, {\em Empirical process technique for dependent
data}, In: H.G. Dehling, T. Mikosch and MSorenson (Eds.), {\em Empirical Process Technique for
 Dependent Data}, p.p. 3--113, Birkh\" auser, Boston, 2002.


\bibitem{Ebe} E. Eberlein, {\em Strong approximation of continuous time stochastic processes},
J. Multivar. Anal. 31 (1989), 220--235.


\bibitem{Fre} M.I. Freidlin, {\em On the factorization of non-negative definite matrices},
Theory Probab. Appl. 13 (1968), 354--356.


\bibitem{GS} U. Gruber and M. Schweizer, {\em A diffusion limit for generalized correlated
random walks}, J. Appl. Probab. 43 (2006), 60--73.


\bibitem{He} H. He, {\em Convergence from discrete-to continuous-time contingent claims
prices}, Review Financial Studies 3 (1990), 523--546.







\bibitem{HK} Ye. Hafouta and  Yu. Kifer, {\em Nonconventional Limit Theorems and
 Random Dynamics}, World Scientific, Singapore, 2018.



\bibitem{IW} N. Ikeda and S. Watanabe, {\em Stochastic Differential Equations
and Diffusion Processes 2nd. ed.}, North-Holland, Amsterdam, 1989.



\bibitem{Kha} R.Z. Khasminskii, {\em A limit theorem for the solution of differential
equations with random right-hand sides}, Theory Probab. Appl. 11 (1966), 390--406.



\bibitem{Ki03} Yu. Kifer,  {\em $L^2$ diffusion approximation for slow motion in averaging},
Stochastics and Dynam. 3 (2003), 213--246.


\bibitem{Ki07} Yu. Kifer, {\em Optimal stopping and strong approximation theorems},
Stochastics 79 (2007), 253--273.


\bibitem{Ki20} Yu. Kifer, {\em Lectures on Mathematical Finance and Related Topics}, World Scientific, Singapore, 2020.


\bibitem{Ki21} Yu. Kifer, {\em Error estimates for discrete approximations
of game options with multivariate diffusion prices}, J. Stoch. Anal. 2 (2021),
no.3, art.8.



\bibitem{KP} J. Kuelbs and W. Philipp, {\em Almost sure invariance principles for partial
sums of mixing $B$-valued random variables}, Annals Probab. 8 (1980), 1003--1036.


\bibitem{KV}
Yu.Kifer and S.R.S Varadhan, {\em Nonconventional
limit theorems in discrete and continuous time via martingales}, Ann.
Probab. 42 (2014), 649-688.


\bibitem{Mao} X. Mao, {\em Stochastic Differential Equations and Applications}, 2nd. ed.,
Woodhead, Oxford, 2010.




\bibitem{MP1} D. Monrad and W. Philipp, {\em Nearby variables with nearby laws and a strong
approximation theorem for Hilbert space valued martingales}, Probab. Th. Rel. Fields 88
(1991), 381--404.


\bibitem{MP2} D. Monrad and W. Philipp, {\em The problem of embedding vector-valued martingales in
a Gaussian process}, Theory Probab. Appl. 35 (1991), 374--377.


\bibitem{PK} G. C. Papanicolaou and W. Kohler, {\em Asymptotic theory of mixing
stochastic ordinary differential equations}, Comm. Pure Appl. Math. 27 (1974),
641--668.


\bibitem{Str} V. Strassen, {\em Almost sure behavior of sums of independent random variables
and martingales}, Proc. Fifth Berkeley Symp. Math. Stat. Probab., II, Part 1, 315--343.


\bibitem{SV} D.W. Stroock and S.R.S. Varadhan, {\em Multidimensional Diffusion processes},
Springer-Verlag, Berlin, 1997.


\bibitem{Zai} A. Yu. Zaitsev, Multidimensional version of a result of Sakharenko
in the invariance principle for vectors with finite exponential moments, I--III,
Theory Probab. Appl. 45 (2001), 624--641; 46 (2002), 490--514, 676--698.

\end{thebibliography}

\end{document}